\documentclass[a4paper]{scrartcl}
\usepackage{amssymb, amsmath, amsthm, mathrsfs, braket}

\usepackage{hyperref}
\usepackage{newtxtext,newtxmath}

\usepackage[all,2cell]{xy}
\objectmargin+{1mm}
\labelmargin+{0.8mm}
\SelectTips{cm}{12}

\usepackage{enumitem}
\setlist[enumerate]{itemsep=0pt,label=$(\mathrm{\roman*})$, topsep=5pt}
\setlist[itemize]{itemsep=0pt, topsep=5pt, labelindent=\parindent,leftmargin=*}
\setlist[description]{itemsep=0pt, topsep=5pt, leftmargin=*}

\newtheorem{thm}{Theorem}[section]
\newtheorem{cor}[thm]{Corollary}
\newtheorem{lem}[thm]{Lemma}
\newtheorem{prop}[thm]{Proposition}

\theoremstyle{definition}
\newtheorem{dfn}[thm]{Definition}
\newtheorem{rem}[thm]{Remark}

\newtheorem*{claim}{Claim}

\usepackage{etoolbox}
\usepackage{framed}
\newcommand\enclosebox[2]{%
  \BeforeBeginEnvironment{#1}{\begin{#2}}%
  \AfterEndEnvironment{#1}{\end{#2}}%
}

\enclosebox{thm}{leftbar}
\enclosebox{cor}{leftbar}
\enclosebox{lem}{leftbar}
\enclosebox{prop}{leftbar}
\enclosebox{dfn}{leftbar}
\enclosebox{rem}{leftbar}
\enclosebox{ex}{leftbar}

\newcommand{\ab}{\mathrm{ab}}
\newcommand{\A}{\mathscr{A}}
\newcommand{\As}{\ol{\A}}
\newcommand{\Ahat}{\widehat{A}}
\newcommand{\Abar}{\ol{A}}

\newcommand{\Cor}{\operatorname{Cor}}
\newcommand{\Cf}{\textit{cf.}\;}
\newcommand{\Coker}{\operatorname{Coker}}
\renewcommand{\div}{\mathrm{div}}
\newcommand{\CH}{\operatorname{CH}_0}

\newcommand{\Ehat}{\widehat{E}}

\newcommand{\F}{\mathbb{F}}
\newcommand{\Fk}{\F}
\newcommand{\FK}{\F_K}
\newcommand{\fin}{\mathrm{fin}}

\newcommand{\geo}{\mathrm{geo}}

\newcommand{\Gal}{\operatorname{Gal}}
\newcommand{\Gm}{\mathbb{G}_{m}}

\newcommand{\Gmhat}{\wh{\mathbb{G}}_m}

\renewcommand{\H}{\mathscr{H}}
\newcommand{\Hn}{\H_n}
\newcommand{\Hhat}{\widehat{\H}}
\newcommand{\Hhatn}{\Hhat_n}

\newcommand{\isomto}{\stackrel{\simeq}{\to}}
\renewcommand{\Im}{\operatorname{Im}}
\newcommand{\inj}{\hookrightarrow}
\newcommand{\ilim}{\varinjlim}
\newcommand{\Id}{\operatorname{id}}

\newcommand{\Jac}{\operatorname{Jac}}
\newcommand{\Jbar}{\ol{J}}

\newcommand{\Ker}{\operatorname{Ker}}
\newcommand{\Kt}{K^{\times}}
\newcommand{\kt}{k^{\times}}

\newcommand{\kx}{k(x)}
\newcommand{\kxt}{\kx^{\times}}
\newcommand{\kbar}{\ol{k}}
\newcommand{\kur}{k^{\ur}}
\newcommand{\K}{\mathscr{K}}
\newcommand{\Kur}{K^{\ur}}
\newcommand{\Kurt}{(\Kur)^{\times}}
\newcommand{\Khat}{\widehat{\K}}
\newcommand{\Kbar}{\overline{\K}}
\newcommand{\Khatn}{\Khat_n}
\newcommand{\Kbarn}{\Kbar_n}

\newcommand{\m}{\mathfrak{m}}

\newcommand{\mK}{\m_K}
\newcommand{\M}{\mathscr{M}}
\newcommand{\Mn}{\M_n}
\newcommand{\Mbar}{\overline{\M}}
\newcommand{\Mhat}{\widehat{\M}}
\newcommand{\Mbarn}{\Mbar_n}
\newcommand{\Mhatn}{\Mhat_n}

\newcommand{\N}{\mathscr{N}}

\newcommand{\ol}[1]{\overline{#1}}
\newcommand{\OK}{O_K}
\newcommand{\Ok}{O_k}
\newcommand{\OKt}{\OK^{\times}}

\newcommand{\opcit}{\textit{op.\,cit.}}
\newcommand{\onto}[1]{\stackrel{#1}{\to}}

\newcommand{\otimesZ}{\otimes_{\Z}}
\newcommand{\otimesM}{\otimes}%{\stackrel{M}{\otimes}}
\newcommand{\Ord}{$(\mathbf{Ord})\,$}

\newcommand{\piab}{\pi_1^{\ab}}
\newcommand{\plim}{\varprojlim}

\newcommand{\Q}{\mathbb{Q}}
\newcommand{\Qp}{\Q_p}

\newcommand{\Ram}{$(\mathbf{Ram})\,$}
\newcommand{\Res}{\operatorname{Res}}

\newcommand{\Rat}{$(\mathbf{Rat})\,$}

\newcommand{\Ramp}{$(\mathbf{Ram}$'$)\,$}

\newcommand{\sn}{\smallskip\noindent}
\newcommand{\shat}{\widehat{s}}
 
\newcommand{\surj}{\twoheadrightarrow}
\newcommand{\ssm}{\smallsetminus}
\newcommand{\Spec}{\operatorname{Spec}}
\newcommand{\Spf}{\operatorname{Spf}}
\newcommand{\SKX}{SK_1(X)}

\newcommand{\Ubar}{\ol{U}}
\newcommand{\U}{\mathscr{U}}
\newcommand{\ur}{\mathrm{ur}}

\newcommand{\V}{\mathscr{V}}
\newcommand{\VX}{V(X)}
\newcommand{\VXdiv}{\VX_{\div}}
\newcommand{\VXfin}{\VX_{\mathrm{fin}} }

\newcommand{\wh}[1]{\widehat{#1}}

\newcommand{\X}{\mathscr{X}}
\newcommand{\Xs}{\ol{\X}}

\newcommand{\Z}{\mathbb{Z}}
\newcommand{\Zhat}{\wh{\Z}}

\title{Galois symbol maps for abelian varieties over a $p$-adic field}
\author{Toshiro Hiranouchi}

\begin{document}
\pagenumbering{arabic}
\maketitle

\begin{abstract}
We study the Galois symbol map associated to 
the multiplicative group and an abelian variety 
which has good ordinary reduction 
over a $p$-adic field. 
As a byproduct, 
one can calculate 
the ``class group'' in the view of the class field theory 
for curves over a $p$-adic field. 
\end{abstract}

\section{Introduction}
\label{sec:intro}
%The aim of this note is to study 
%behavior of 
%the ``class group'' in the class field theory of curves 
%over a $p$-adic field. 
%To state our results precisely, 
Let $k$ be a \textbf{$p$-adic field}, that is, a finite extension of $\Qp$ and 
its residue field is denoted by $\Fk$.  
An objective of the class field theory of 
a (projective smooth and geometrically connected) 
curve $X$ over $k$ with function field $k(X)$ 
(\cite{Blo81}, \cite{Sai85a}) % and \cite{Yos03} 
is to describe 
the abelian fundamental group $\piab(X)$ by 
an abelian group 
\[
  SK_1(X) = \Coker\left(K_2(k(X)) \to \bigoplus_{x\in X_0} \kxt  \right)
\]
(for the precise definition of $\SKX$, see \eqref{def:SK1})
through 
the reciprocity map 
\[
\rho:\SKX \to \piab(X) 
\] 
(\Cf Thm.\ \ref{thm:cft}). 
 %It is known (\Cf Thm.\ \ref{thm:cft}) that 
%\begin{itemize}[label=$\circ$] 
%	\item $\Ker(\rho)$ is the maximal divisible subgroup of $\SKE$, and 
%	\item $\piab(E)/\ol{\Im(\rho)} \simeq \Zhat^r$, 
%	where $\ol{\Im(\rho)}$ is the topolobical closure of the image $\Im(\rho)$ 
%	and 
%	the invariant $r = r(E)$ 
%	is called the \textbf{rank} of $E$ (\Cf \cite{Sai85a}, Def.\ 2.5) 
%	which 
%	depends on the type of the reduction of $E$. 
%	For example, $r=0$ if $E$ has good reduction. 
%\end{itemize}
%
The ``geometric part'' 
$\piab(X)^{\geo} := \Ker\left(\piab(X) \to \piab(\Spec(k))\right)$ 
is approximated by 
\[
\VX = \Ker\left(\partial:\SKX \to k^{\times}\right),
\] 
where $\partial$ is defined by the tame symbols. 
It is known that the induced map $\tau: \VX \to \piab(X)^{\geo}$ from $\rho$  
has the finite image 
and the kernel is the maximal divisible subgroup $\VXdiv$ of $\VX$. 
In particular, we have a decomposition 
\[
	\VX = \VXdiv \oplus \VXfin,
\] 
for some finite group $\VXfin$. 
%The group $\VXfin$  
%is an analogue of the class group  
%in the classical class field theory. 
Now we further assume that 
$X$ has \emph{good reduction and also $X(k)\neq \emptyset$}. 
%This is just the situation that Bloch studied.  
Recall that $X$ said to have \textbf{good reduction} if 
the special fiber $\Xs := \X\otimes_{\Ok}\Fk$ of the regular model $\X$ over $\Ok$ with $\X\otimes_{\Ok} k \simeq X$ 
is also a smooth curve over $\Fk$. 
The Jacobian variety $\Jac(\X)$ associated to $\X$ has the generic fiber $J = \Jac(X)$ and the special fiber $\Jbar  = \Jac(\Xs)$. 
	In particular, $J$ has good reduction. 
	%(Cf Milne, Jaciban var. Proof of Cor 12.3)
This group $\VXfin$ is related to 
the $G_k$-coinvariant part of the Tate module 
$T(J) = \plim_{m}J[m]$
as 
\[
	\VXfin \isomto T(J)_{G_k}.
\]
In this setting, 
the prime to $p$-part of $\VXfin$ is known well 
as $\VX/m \simeq \Jbar(\Fk)/m$ 
for any $m$ prime to $p$ (by S.~Bloch, see Prop.~\ref{prop:Bloch}). 
For the $p$-part, 
only the finiteness is proved (\cite{Blo81}, Prop.~2.4).
The aim of this note is \emph{to study more explicitly 
on this $\VXfin$ by determining 
the group structure} 
under the following conditions: 
\begin{itemize}
	\item[] \Rat\ $J[p]\subset J(k)$,  
	\item[] \Ord\ $\Jbar$ has ordinary reduction, and % [Maz] Lem.~4.27
	\item[] \Ram\ $k(\mu_{p^{N+1}})/k$ is 
	a non-trivial totally ramified extension, where 
	 $\mu_{p^{N+1}}$ is the group of $p^{N+1}$-th roots of unity,  and 
	 $N = \max\set{ n | J[p^n] \subset J(k)}$.
\end{itemize}
% X has good red -> J has good red. but the converse does not hold.
% We may only assume J has good ordinary reduction 
% if we generalize Thm. \ref{thm:main} to the prime to p-part  
% avoiding Bloch's thm Prop.~\ref{prop:Bloch}.
The main result of this note is the following: 

\begin{thm}[Cor.~\ref{cor:main}]
	\label{thm:intro}
%	Let $X$ be a smooth and projective curve over $k$ with $X(k)\neq \emptyset$. 
	Under the conditions 
	\Rat\!,\ \Ord\!,\ and \Ram\ for $J = \Jac(X)$ as above, we have 
	\[
	\VXfin \simeq (\Z/p^{N})^{\oplus g} \oplus \Jbar(\Fk), 
	\]
	where $g = \dim J$. 
\end{thm}
%Note that $\Jbar(\Fk)$ plays a role of 
%the class group in the class field theory of the special fiber $\Xs$ which is 
%a curve over a \emph{finite field}. 
Since the divisible part $\VXdiv$ is the kernel of the 
reciprocity map $\tau:V(X)\to \piab(X)^{\geo}$ 
and $\tau$ is surjective in the case where $X$ has good reduction, 
we obtain the structure of the geometric part of the fundamental group as 
\[
\piab(X)^{\geo}\simeq (\Z/p^{N})^{\oplus g} \oplus \overline{J}(\Fk).
\]
The key tool to compute the group $\VX$ 
is the so called \textbf{Galois symbol map}  
of the following form 
\[	
s_m: K(k;\Gm,J)/m \to H^2(k,\mu_{m}\otimes J[m]), 
\]
where $K(k;\Gm,J)$ is the Somekawa $K$-group 
which is isomorphic to $\VX$ 
and the map $s_m$ is constructed by 
a similar way as the Galois symbol map 
on the Milnor $K$-group 
$K_2(k)/m\to H^2(k,\mu_m^{\otimes 2})$ 
(for the precise construction of the map $s_m$, 
see Definition \ref{def:symbol}).

Finally, we give some remarks on the conditions above. 
%On the condition \Rat, 
%consider an elliptic curve $X= E$ over $k$ with 
%good ordinary reduction. 
%Note that  we have $\Jac(E) = E$. 
%When $E[p]\not\subset E(k)$,  
%we show that $\VEfin = 0$ under some conditions 
%(Prop.~\ref{prop:triv}). 
The condition \Ram is technical. 
For example, let us consider an elliptic curve $X = E$ over $k = \Qp(\mu_{p^M})$ for some $M\ge 1$ assuming that $E$ 
has good ordinary reduction as in \Ord. 
Note that we have $\Jac(E) = E$. 
From the Weil pairing, we have an inequality $N\le M$ in general (\Cf \cite{Sil106}, Chap.~III, Cor.~8.1.1), 
where $N = \max \set{n | E[p^n] \subset  E(k)}$. 
Thus, if $E[p^M] \subset E(k)$ then 
$N = M$. Consequently, the condition \Ram above automatically holds. 
For the case $M=1$ 
(from the above argument, \Rat implies \Ram\!)
and $p=3$, 
by computing SAGE \cite{sage}, 
there are 683 elliptic curve $E_0$ over $\Q$ 
with good ordinary reduction at $p = 3$, 
$\ol{E_0}(\F_3)[3]\neq 0$ and the conductor $< 1000$. 
Among them, 269 curves satisfy 
$E[3]\subset E(k)$ and hence the condition \Ram\!, 
where $E = E_0 \otimes_{\Q}k$ is the base change of $E$ to $k=\Q_3(\mu_3)$. 

%% Contents
\subsection*{Contents}
The contents of this note is the following:
\begin{itemize}[label=$\circ$]
\item Section \ref{sec:cft}: 
After recalling the definition of the Mackey functor and that of the product for two Mackey functors,   
	we review the class field theory for curves over $p$-adic fields following Bloch \cite{Blo81} and Saito \cite{Sai85a}.
\item Section \ref{sec:Kummer}: 
We determine the image of the Kummer map $A(k)/p^n \to H^1(k,A[p^n])$ (\Cf \eqref{def:Kummer}) associated to an abelian variety $A$ with good ordinary reduction over $k$ (Prop.~\ref{prop:Tak}). 
This extends the main theorem in \cite{Tak11} for an elliptic curve. The proof is essentially same as in \opcit\ 
\item Section \ref{sec:symbol}: We show that the Galois symbol map (defined in Def.~\ref{def:symbol}) 
	associated to the multiplicative group $\Gm$ and an abelian variety $A$ over a $p$-adic field is bijective (Thm.~\ref{thm:GmAV}), and 
	determine the group structure of the Somekawa $K$-group $K(k;\Gm,A)$ (Thm.~\ref{thm:main}). 
	Using the isomorphism $K(k;\Gm,J) \simeq V(X)$ (\Cf \eqref{eq:VX}) for a curve $X$ as in Theorem \ref{thm:intro} and $J = \Jac(X)$, 
	we obtain the structure of $\VXfin$. 
\end{itemize}

%% Notation
\subsection*{Notation}
Throughout this note,  
we use the following notation 
\begin{itemize}
	\item $k$: a finite extension of $\Qp$, 
	\item $\Fk$: the residue field of $k$, and  
%	\item $k^{\ur}$: the maximal unramified extension of $k$, 
	\item $G_k = \Gal(\kbar/k)$: the absolute Galois group of $k$.
\end{itemize}
For a finite extension $K/k$, we define 
\begin{itemize}
	\item $O_K$: the valuation ring of $K$, 
	\item $\FK$: the residue field of $K$, 
	\item $U_K = \OK^{\times}$: the unit group, and 
	\item $U_K^i = 1 + \mathfrak{m}_K^i$: the higher unit group. 
\end{itemize}
For an abelian variety over $k$, we define 
\begin{itemize}
%	\item $G_k^{\ab} = \Gal(k^{\ab}/k)$: the Galois group of the maximal abelian extension $k^{\ab}$ of $k$.
	\item $\pi:A\to \Abar$: the reduction map, 
	\item $\Ahat$\,: the formal group law of $A$  (\Cf \cite{HS00}, Sect.~C.2), and 
	\item $A_1 := \Ker(\pi:A\to \Abar)$\,: the kernel of the reduction map $\pi$.
\end{itemize}
For an abelian group $G$ and $m\in \Z_{\ge 1}$,  
we write $G[m]$ and $G/m$ for the kernel and cokernel 
of the multiplication by $m$ on $G$ respectively. 

\subsection*{Acknowledgments}
I would like to thank Professor Evangelia Gazaki and Dr.~Isable Leal for their 
valuable comments and encouragement.
This work was supported by KAKENHI 25800019.

\section{Class field theory}
\label{sec:cft}

\subsection*{Mackey functors}

\begin{dfn}[\Cf \cite{RS00}, Sect.~3]
\label{def:Mack}
    A \textbf{Mackey functor} $\M$ (over $k$)  (or a \textbf{$G_k$-modulation} in the sense of \cite{NSW08},  Def.~1.5.10) 
    is a contravariant 
    functor from the category of \'etale schemes over $k$ 
    to the category of abelian groups 
    equipped with a covariant structure 
    for finite morphisms 
    such that 
    $\M(X_1 \sqcup X_2)  = \M(X_1) \oplus \M(X_2)$ 
    and if 
    \[
    \xymatrix@C=15mm{
      X' \ar[d]_{f'}\ar[r]^{g'} & X \ar[d]^{f} \\
      Y' \ar[r]^{g} & Y
    }
    \]
    is a Cartesian diagram, then the induced diagram 
    \[
    \xymatrix@C=15mm{
      \M(X') \ar[r]^{{g'}_{\ast}} & \M(X)\\
      \M(Y') \ar[u]^{{f'}^{\ast}} \ar[r]^{{g}_{\ast}} & \M(Y)\ar[u]_{f^{\ast}}
    }
    \]
    commutes. 
\end{dfn}

For a Mackey functor $\M$, 
we denote by $\M(K)$ 
its value $\M(\Spec(K))$  
for a field extension $K$ of $k$.
For any finite extensions $k\subset K \subset L$, 
the induced homomorphisms 
from the canonical map $j:\Spec(L) \to \Spec (K)$
are denoted by 
\[
N_{L/K}:=j_{\ast}:\M(L) \to \M(K),\quad \mbox{and}\quad  
	 \Res_{L/K}:=j^{\ast}:\M(K)\to \M(L).
\]
The category of Mackey functors over $k$ forms 
an abelian category with the following tensor product: 
%In particular, 
%for each $m$, the multiplication by $m$ gives 
%\[
%\M[m] := \Ker(m:\M \to \M),\quad \M/m:=\Coker(m:\M\to \M).
%\]

\begin{dfn}[\Cf \cite{Kahn92}]
	\label{def:otimesM}
For Mackey functors $\M$ and $\N$, 
their \textbf{Mackey product} 
$\M\otimesM \N$ 
is defined as follows: 
For any finite field extension $k'/k$, 
\begin{equation}
	\label{eq:otimesM}
(\M\otimesM \N ) (k') := 
\left.\left(\bigoplus_{K/k':\,\mathrm{finite}} \M(K) \otimesZ \N(K)\right)\middle/\ \textbf{(PF)},\right. 
\end{equation}
where \textbf{(PF)} stands for the subgroup generated 
by elements of the following form: 
\begin{itemize}
	\item [\textbf{(PF)}]
For finite field extensions 
$k' \subset K \subset L$, 
\begin{align*}
&N_{L/K}(x) \otimes y - x \otimes \Res_{L/K}(y) \quad 
\mbox{for $x \in \M(L)$ and $y \in \N(K)$, and}\\
&x \otimes \Res_{L/K}(y) - N_{L/K}(x) \otimes y \quad 
\mbox{for $x \in \M(K)$ and $y \in \N(L)$}.
\end{align*}
\end{itemize}
\end{dfn}

For the Mackey product  
$\M\otimesM \N$,  
we write $\set{x,y}_{K/k'}$ 
for the image of 
$x \otimes y \in 
\M(K) \otimesZ \N(K)$ in the product 
$(\M\otimesM \N)(k')$. 
For any finite field extension $k'/k$, 
%the pull-back 
%\[
%  \Res_{k'/k} := j^{\ast}: \left(\M\otimes N \right)(k) \longrightarrow 
%\left(\M\otimes \N\right)(k').
%\] 
the push-forward 
\begin{equation}
	\label{eq:norm2}
  N_{k'/k} = j_{\ast}: (\M \otimesM \N )(k') \longrightarrow 
(\M\otimesM \N) (k)
\end{equation}
is given by 
$N_{k'/k}(\set{x,y}_{K/k'}) = \set{x,y}_{K/k}$. 
For each $m \in \Z_{\ge 1}$, we define a Mackey functor  $\M/m$ 
by 
\begin{equation}
	\label{def:quot}
(\M/m) (K) := \M(K)/m
\end{equation}
for any finite extension $K/k$. 
We have 
\[
(\M/m \otimesM \N/m )(k) \simeq(\M\otimesM \N)(k)/m \quad ( =  ((\M\otimesM \N)/m)(k) \   \mbox{in the sense of \eqref{def:quot}}).
\]

Every $G_k$-module $M$ defines a Mackey functor 
defined by the fixed sub module $M(K) := M^{\Gal(\kbar/K))}$ denoted by $M$. Conversely, 
assume a Mackey functor $\M$ satisfies the  \textbf{Galois descent},  
meaning that,  for every finite Galois extensions $L/K$, 
the restriction 
\[
\Res_{L/K}: \M(K) \isomto \M(L)^{\Gal(L/K)}
\]
is an isomorphism. This Mackey functor $\M$ gives 
a $G_k$-module $\M = \ilim_{K/k}\M(K)$ (\Cf \cite{NSW08}, Chap.~1, Sect. 5, Ex.~1) denoted again by $\M$. 
For any $m\in \Z_{\ge 1}$, 
the connecting homomorphism associated to the short exact sequence 
$0\to \M[m]\to \M \onto{m} \M \to 0$ as $G_k$-modules gives 
\begin{equation}
\label{def:Kummer}
	\delta_{\M}: \M(K)/m \inj H^1(K,\M[m])
\end{equation}
which is often called the \textbf{Kummer map}. 

\begin{dfn}[\Cf \cite{Som90}, Prop.~1.5]
\label{def:symbol}
For Mackey functors $\M$ and $\N$ with Galois descent,  
the \textbf{Galois symbol map} 
\begin{equation}
	\label{def:sM}
s_m^M:(\M\otimesM \N)(k)/m \to H^2(k,\M[m]\otimes \N[m])
\end{equation}
is defined by the cup product and the corestriction 
as follows:  
 \[
s_m^M(\set{x,y}_{K/k}) = \mathrm{Cor}_{K/k}\left(\delta_{\M}(x)\cup \delta_{\N}(y)\right).
\]
\end{dfn}

\subsection*{Somekawa $K$-group}
For two semi-abelian varieties $G_1$ and $G_2$ over $k$, 
the $G_k$-modules $G_1(\kbar)$  and $G_2(\kbar)$ 
define a Mackey functors 
with Galois descent which we denote also by $G_1$ and $G_2$. 
The \textbf{Somekawa $K$-group} $K(k;G_1, G_2)$
is a quotient of  
$(G_1\otimesM G_2)(k)$ (for the definition, 
see \cite{Som90}, \cite{RS00}). 
The Galois symbol map 
$s_m^M:(G_1\otimesM G_2)(k)/m\to H^2(k,G_1[m]\otimes G_2[m])$ (Def.~\ref{def:symbol}) 
factors through $K(k;G_1,G_2)$ and the 
induced map 
\begin{equation}
	\label{def:s}
	s_m:K(k;G_1, G_2)/m \to H^2(k,G_1[m]\otimes G_2[m])
\end{equation}
is also called the \textbf{Galois symbol map}. 
Somekawa presented a ``conjecture''  
in which the map $s_m$ is injective 
(for arbitrary field). 
For the case $G_1 = G_2 = \Gm$, 
as $K(k;\Gm,\Gm) \simeq K_2^{M}(k)$  
the conjecture holds by 
the Merkurjev-Suslin theorem (\cite{MS82}). 
Although it holds in some special cases (\cite{Yam05}, \cite{Yam09}, and \cite{RS00}),  
Spie\ss\ and Yamazaki disproved this for some tori (\cite{SY09}, Prop.\ 7).

\subsection*{Class field theory}
Following \cite{Blo81}, \cite{Sai85a}, we 
recall the class field theory for a curve over $k$. 
Let $X$ be a projective smooth and geometrically connected curve over $k$. 
Define 
\begin{itemize}
	\item $X_0$: the set of closed points in $X$,  
	\item $k(X)$: the function field of $X$, 
	\item $k(x)$: the residue field at $x\in X_0$, and 
	\item $k(X)_x$: the completion of $k(X)$ at $x\in X_0$.
\end{itemize}
We define   
\begin{equation}
	\label{def:SK1}
  SK_1(X) := \Coker\left(\partial:K_2(k(X)) \to \bigoplus_{x\in X_0} \kxt  \right),
\end{equation}
where the map $\partial$  is given by 
the direct sum of the boundary map $K_2(k(X)_x)\to K_1(k(x)) = \kxt$ for $x\in X_0$. 
Note that the residue field $\kx$ is a finite extension field of $k$ so that $\kx$ is also a $p$-adic field. 
The reciprocity maps $\kxt \to \piab(x)$ 
of the local class field theory  
of $\kx$ for $x\in X_0$ 
induce the \textbf{reciprocity map} 
\[
\rho:SK_1(X) \to \piab(X). %(\Cf\cite{Blo81}, \cite{Sai85a})
\]
The map $\rho$ is compatible with 
the reciprocity map $\rho_k:\kt \to G_k^{\ab} = \Gal(k^{\ab}/k) =\pi_1^{\ab}(\Spec(k))$ of the base field $k$ 
as in the following commutative diagram: 
\begin{equation}
\label{eq:rho}
	\vcenter{
	\xymatrix{
	0 \ar[r]& V(X)\ar[d]^{\tau} \ar[r] &SK_1(X) \ar[r]^-{N}\ar[d]^{\rho} & \kt \ar[r]\ar[d]^{\rho_k} & 0\,\\
	0  \ar[r]& \piab(X)^{\geo} \ar[r]& \piab(X)\ar[r]^{\varphi} & G_k^{\ab}\ar[r] & 0,
	}}
\end{equation}
where 
%$G_k^{\ab} = \Gal(k^{\ab}/k) =\pi_1^{\ab}(\Spec(k))$ is 
%the Galois group of the maximal abelian 
%extension $k^{\ab}$ of $k$, 
the map $\varphi$ is induced from the structure map $X\to \Spec(k)$, 
$N$ is 
induced from the norm maps $\kxt \to \kt$ for each $x\in X_0$, 
and the groups $V(X)$ and $\piab(X)^{\geo}$ are defined by 
the exactness. 
The main theorem of the class field theory for $X$ is the following:  

\begin{thm}[\cite{Blo81}, \cite{Sai85a}]
\label{thm:cft}
\begin{enumerate}
	\item $\piab(X)/\ol{\Im(\rho)} \simeq \Zhat^{\oplus r}$ for some $r\ge 0$, 
	where $\ol{\Im(\rho)}$ is the topological closure of the image $\Im(\rho)$ in $\piab(X)$.
	\item $\Ker(\rho) = SK_1(X)_{\div}$, 
	where $SK_1(X)_{\div}$ is the maximal divisible subgroup of $SK_1(X)$. 
	\item $\Ker(\tau) = \VXdiv$, 
	where $\VXdiv$ is the maximal divisible subgroup of $V(X)$. 
	\item $\Im(\tau)$ is finite.
\end{enumerate}
\end{thm}
Note that the invariant $r$ above is determined by 
the special fiber of the N\'eron model of the Jacobian variety $J = \Jac(X)$. 
In particular, we have $r=0$ if $X$ has good reduction. 
From the above theorem, the group $\VX$ has a decomposition 
\[
V(X) =V(X)_{\div}\oplus V(X)_{\fin},
\]
where the reduced part $V(X)_{\fin}$ is a finite subgroup. 

From now on, we \emph{assume} that $X(k)\neq \emptyset$. 
The geometric fundamental group 
$\piab(X)^{\geo}$ is written 
by the Tate module 
$T(J) = \plim_{m}J[m]$ of the Jacobian variety $J= \Jac(X)$ as 
\[
\piab(X)^{\geo} \simeq T(J)_{G_k}, 
\] 
where 
$T(J)_{G_k}$ is the $G_k$-coinvariant quotient of $T(J)$ (\cite{Sai85b}, Chap.~II, Lem.~3.2). 
On the other hand, 
the group $V(X)$ is written by the Somekawa $K$-group as  
\begin{equation}
	\label{eq:VX}
	 V(X) \simeq  K(k; \Gm, J)
\end{equation} 
(\cite{Som90}, Thm.~2.1, \cite{RS00}, Rem.~2.4.2 (c)). 
The reciprocity map $\tau$ in the diagram \eqref{eq:rho}
coincides with the Galois symbol map 
associated with $\Gm$ and $J$ as 
in the following commutative diagram: 
\begin{equation}
	\label{diag:sm}
	\vcenter{\xymatrix{
	V(X)/m \ar[d]^{\simeq} \ar@{^{(}->}[r]^-{\tau} & \ar[d]^{\simeq} \piab(X)^{\geo}/m\\
	K(k;\Gm,J)/m\ar[r]^-{s_m} & H^2(k,\mu_{m}\otimes J[m]).
	}}
\end{equation}
Here, the right vertical map is induced from 
$H^2(k,\mu_m\otimes J[m]) \simeq J[m]_{G_k}$ 
which is given by the local Tate duality theorem. 
From the construction, 
the composition $V(X)/m \simeq K(k;\Gm,J)/m \onto{s_m} H^2(k,\mu_m\otimes J[m])$ 
is compatible with the map 
$V(X)\to H^2(k,\mu_m\otimes J[m])$ 
given by Bloch (\cite{Blo81}, Thm.~1.14). 
%Note also that the Galois symbol map $s_m$ above 
%is known to be injective (for an arbitrary field) by Spie\ss\ (\Cf \cite{Yam05}, Sect.~6).
From the diagram above, we have $ V(X)/m \simeq  \Im(s_m)$. 
For the prime to $p$-part of $\VX$, we have the following proposition: 

\begin{prop}[\cite{Blo81}, Prop.~2.29]
\label{prop:Bloch}
	Assume that $X$ has good reduction and $X(k)\neq \emptyset$. 
	Then, for any $m\in \Z_{\ge 1}$ prime to $p$, 
 we have $\VX/m \simeq \Jbar(\Fk)/m$, 
 where $\Jbar$ is the reduction of $J$. 
\end{prop}

Note that if $X$ has good reduction, then 
the Jacobian variety $J= \Jac(X)$ has also good reduction. 
But, the converse does not hold.

\subsection*{Prime to $p$-part}
We extend Proposition \ref{prop:Bloch} 
using the Somekawa $K$-group 
to an abelian variety following
Bloch's proof of the above proposition essentially. 

\begin{prop}
\label{prop:Bloch2}
	Let $A$ be an abelian variety over $k$ 
	which has good reduction. 
	For any $m\in \Z_{\ge 1}$ prime to $p$, 
	we have $K(k;\Gm,A)/m\simeq \Abar(\F)/m$.
	\end{prop}
\begin{proof}
	The kernel of the reduction $\pi:A(k)\to \Abar(\F)$ 
	is isomorphic to the formal group $\Ahat(\m_k) =:\Ahat(k)$ which has no prime to $p$ torsion (\cite{HS00}, Prop.~C.2.5, Thm.~C.2.6). 
	The reduction map induces $A(k)/m \simeq \Abar(\F)/m$ 
	for each $m$ prime to $p$.

	First, we define $\psi:A(k)/m \to K(k;\Gm,A)/m$ 
	by 
	\[
	\psi(x) = \set{\pi,x}_{k/k},
	\]
	where $\pi$ is a uniformizer of $k$. 
	The map does not depend on the choice of $\pi$. 
	In fact, for any $u\in \Ok^{\times}$, 
	by taking $\xi \in [m]^{-1}(x)$, $K = k(\xi)/k$ is unramified. 
	By local class field theory, 
	there exists $\mu \in K^{\times}$ such that $N_{K/k}\mu = u$. We have  
	$\set{u,x}_{k/k} = \set{N_{K/k}\mu, x}_{k/k} = 
	\set{\mu, \Res_{K/k}x}_{K/k} = \set{\mu, m\xi}_{k/k} = 0$ in $K(k;\Gm,A)/m$.
	
	\begin{claim}
	$\psi$ is surjective. 	
	\end{claim}
  	\begin{proof}
  		Take an element of the form $\set{\mu\varpi^n,\xi}_{K/k}$ 
  		in $K(k;\Gm,A)/m$, 
  		where $\varpi$ is the uniformizer in $K$, and 
  		$\mu\in \OKt$. 
  		As above, we have $\set{\mu,\xi}_{K/k}=0$ by considering 
  		the unramified extension $K([m]^{-1}\xi)/K$. 
  		It is enough to show that the element $\set{\varpi,\xi}_{K/k}$ is generated by the elements of the form $\set{\pi,x}_{k/k}$. 
  		
  		Let $k'$ be the maximal unramified subextension of $K/k$. 
  		$\pi' = \Res_{k'/k}\pi$ is also a uniformizer 
  		of $k'$. 
  		The extension $K/k'$ is totally ramified 
  		so that we may take $\varpi$ as  
  		$N_{K/k'}\varpi = \pi'$. 
  		Since the restriction $\Res_{K/k'}:A(k')/m \isomto A(K)/m$ is bijective, 
  		there exists $x'\in A(k')/m$ such that $\Res_{K/k'}x' = \xi$. 
  		Therefore
  		\begin{align*}
  		\set{\varpi,\xi}_{K/k} &= \set{\varpi,\Res_{K/k'}x'}_{K/k} \\
  		&=\set{N_{K/k'}\varpi, x'}_{k'/k}  \quad(\mbox{by the projection formula})\\ 
		&= \set{\Res_{k'/k}\pi, x'}_{k'/k}\\
		&= \set{\pi, N_{k'/k}x'}_{k/k}\quad(\mbox{by the projection formula}).
  		\end{align*}
  		From these equalities, the map $\psi$ is surjective.
  	\end{proof}
	Next, we show the map $\psi$ is bijective by 
	showing the following claim. 
	
	\begin{claim}
	For any $m$ prime to $p$, we have
	$\# K(k;\Gm,A)/m \ge \# \Abar(\F)/m$.
	\end{claim}
	\begin{proof}	
  	The direct limit of the Galois symbol map 
  	\[
  	\plim s_m: K(k;\Gm,A) \to \plim_{m\ge 1} H^2(k,\mu_m\otimes A[m]) \simeq T(A)_{G_k}
  	\] 
  	is known to be surjective (\cite{Gaz19}, Thm.~A,1), 
  	where $T(A)_{G_k}$ is the $G_k$-coinvariant quotient of the Tate module $T(A) = \plim_mA[m]$  of $A$ 
  	and the latter isomorphism follows from the local Tate duality theorem. 
  	Write $T(A) = T_p(A)\times T'(A)$, where 
  	$T'(A) = \plim_{(m,p) = 1}A[m]$. 
  	As $A$ has good reduction, the inertia subgroup $I \subset G_k$ acts trivially on $T'(A)$ 
  	so that $T'(A)_{G_k} \simeq T'(\Abar)_{G_{\F}}$, 
  	where $T'(\Abar) = \plim_{(m,p) = 1}\Abar[m]$. 
  	The Weil conjecture for abelian varieties implies 
  	that $T'(\Abar)_{G_{\F}}$ 
  	is isomorphic to the prime to $p$-part of the torsion subgroup of $\Abar(\F)$ (\Cf \cite{Blo81}, Prop.~2.4; \cite{KL81}, Thm.~1 (ter)). 
  	For any $m$ prime to $p$, we have 
  	\[
  	\# K(k;\Gm,A)/m \ge \# T(A)_{G_k}/m = \# T'(\Abar)_{G_{\F}}/m =\# \Abar(\F)/m.
  	\]
	\end{proof}
  	From the above claims, the surjective homomorphism 
  	\[
  	\Abar(\F)/m \simeq A(k)/m \stackrel{\psi}{\surj} 
  	K(k;\Gm,A)/m
  	\]
  	is bijective by comparing the cardinality.
\end{proof}

%%%%%%%%%%%%%%%%%%%%%%%%%%%
\section{Kummer map}
\label{sec:Kummer}
In this section, let $A$ be an abelian variety of dimension $g$ over $k$ 
	\textit{assuming} 
	\begin{itemize}
	\item[] \Ord\ $A$ has good ordinary reduction,   
	in the sense that $A$ has good reduction and its reduction  
	$\Abar$ has ordinary reduction,  % [Maz] Lem.~4.27
	and 
	\item[] \Rat\ $A[p]\subset A(k)$.
\end{itemize} 
Let $\kur$ be the completion of the maximal unramified 
	extension of $k$.
%%
%Recall that 
%$A$ has \textbf{ordinary good reduction} 
%if its reduction $\ol{A}$ (over $\F_k$) satisfies 
%	$\ol{A} [p^n] \simeq (\Z/p^n)^{\oplus g}$ for any $n\in \Z_{\ge 1}$. 
%We denote by $\A$ its N\'eron model over $\Ok$. 
%Let $\kur$ be the completion of the maximal unramified 
%	extension of $k$. 
%	$A_1[p^n] = A_1(K)[p^n]$. 
	The kernel of the reduction is 
	\[
	A_1(\kur) = \Ker\left(\pi:A(\kur) \to \Abar(\Fk)\right) 
	\simeq \Ahat(\m_{\kur}) =: \Ahat(\kur)
	\] 
	(\Cf \cite{HS00}, Thm.~C.2.6). 
	It is known that we have $\Ahat \times_{\Ok}\Spf(O_{\kur}) \simeq (\Gmhat)^{\oplus g}$,  
	where $\Gmhat$ is the multiplicative group 
	(\cite{Maz72}, Lem.~4.26, Lem.~4.27). 
	Since we have $A[p]\subset A(k)$, 
	$A_1[p] \subset A_1(k)$ and hence 
	we obtain isomorphisms 
	\begin{equation}
	\label{eq:A1}
			A_1[p] = A_1(\kur)[p] \simeq \Ahat(\kur)[p] \simeq \left((\Gmhat)(\kur)[p]\right)^{\oplus g} \simeq (\mu_{p})^{\oplus g}. 
	\end{equation}
%Since we assumed that $A$ has good ordinary reduction \Ord, 
%we have $\ol{A}[p^n] \simeq (\Z/p^n)^{\oplus g}$. 
%%The assumption \Rat implies that 
%For any  $n\le N$, 
%we have the short exact sequence 
%\[
%0 \to \Ahat[p^n] \to A[p^n] \to \Abar[p^n] \to 0
%\]
%as (trivial) Galois modules.
%splits so that we have 
%$\Ahat[p^n] \simeq (\Z/p^n)^{\oplus g}$. 
%
%the $p^n$-torsion subgroup scheme $\A[p^n]$ %$(= \A^{\circ}[p^n]_f)$ 
%of the N\'eron model $\A$ is finite and flat over $\Ok$. % [GazakiII] p. 5
%Now, we consider the formal group $\Ahat$ of $A$ (\cite{HS00}, Sect.~C.2). 
%From the assumption \textbf{(Rat)}, we have $\Ahat[p^n] \subset \Ahat(\Ok)$.
%Recall that 
%As we have $A[p^n] \simeq (\Z/p^n)^{\oplus 2g}$ as abelian groups. 
%%(\cite{HS}, Sect.~A.7.2) 
%By using the logarithm, we have 
%$\Ahat[p^n] \simeq (\Z/p^n)^{\oplus g}$ (****)
%%$\log:\Ahat(\Ok) \isomto \mathrm{Lie}(\A/\Ok)\otimes_{\Ok} \m_k$  
%%Cf Katz, Galois properties ... Appendix 
%%Lie algebras and smoothness of group schemes
%%\url{http://www.math.columbia.edu/~chaoli/docs/AbelianVarieties.html#sec10}
%Now, we fix an isomorphism 
%\begin{equation}
%	\label{eq:A1}
%\Ahat[p^n] \isomto (\mu_{p^n})^{\oplus g}
%\end{equation}
%of (trivial) Galois modules and 
Now we choose an isomorphism 
\begin{equation}
	\label{eq:isom}
	A[p] \isomto (\mu_{p})^{\oplus 2g}
\end{equation}
of (trivial) Galois modules 
which makes  
the following diagram commutative: 
\[
\xymatrix@C=15mm{
A_1[p] \ar@{^{(}->}[r] \ar[d]^{\simeq} & A[p]\ar[d]^{\simeq} \\
(\mu_{p})^{\oplus g}  \ar@{^{(}->}[r]^-{(\Id,1)} & (\mu_{p})^{\oplus g}\oplus (\mu_{p})^{\oplus g},
}
\]
where 
the left vertical map is given in \eqref{eq:A1}, 
and the bottom horizontal map is defined by 
\[
(\mu_{p})^{\oplus g} \to (\mu_{p})^{\oplus 2g}; (x_1,\ldots , x_g) \mapsto (x_1,\ldots , x_g, 1,\ldots ,1).
\] 

In the following, the Kummer map on $\Gm$
gives the isomorphism $\delta_{\Gm} :\Kt/p \isomto H^1(K,\mu_{p})$ 
for an extension $K/k$ and we identify these groups. 
The fixed isomorphism \eqref{eq:isom} induces an isomorphism ($\clubsuit$) below 
\[
\delta_{A}^K : A(K)/p \stackrel{\delta_A}{\inj}  H^1(K,A[p]) \stackrel{(\clubsuit)}{\simeq} H^1(K,\mu_{p})^{\oplus 2g} = (\Kt/p)^{\oplus 2g}.
\]

\begin{prop}
\label{prop:Tak}
For any finite extension $K/k$, we have the following:
\begin{enumerate}
	\item 
The image of the Kummer map $\delta_A^K$ equals to 
	\[
	(\Ubar_K)^{\oplus g} \oplus \Ker\left(\Kt/p \onto{j} (\Kur)^{\times}/p\right)^{\oplus g}, 
	\] 
	where $\Ubar_K := \Im(U_K \to \Kt/p)$, 
	$\Kur$ is the completion of the maximal unramified extension of $K$, 
	and $j$ is the map induced from the inclusion $\Kt \inj \Kurt$. 
\item
	The image of the composition $A_1(K)/p\to A(K)/p \onto{\delta_A^K} (\Kt/p)^{\oplus 2g}$ coincides with $(\Ubar_K)^{\oplus g}$. 
	In particular, $\Ahat(K)/p \simeq A_1(K)/p \simeq (\Ubar_K)^{\oplus g}$.
\end{enumerate}
\end{prop}

In the following, 
we fix a finite extension $K/k$ 
and prove the above proposition. 
	First, we show the following lemma on 
	$\delta_{A_1}^{\Kur}: A_1(\Kur)/p \inj H^1(\Kur,A_1[p])\stackrel{(\diamondsuit)}{\simeq} (\Kurt/p)^{\oplus g}$, 
where the isomorphism $(\diamondsuit)$ is given by \eqref{eq:A1}.

\begin{lem}
\label{lem:Tak1}
\begin{enumerate}
	\item 
		$\Im(\delta_{A_1}^{\Kur}) \subset (\Ubar_{\Kur})^{\oplus g}$.
	\item 
	$\Im(\delta_A^{\Kur}) \subset (\Ubar_{\Kur})^{\oplus g}$.
\end{enumerate}
\end{lem}
\begin{proof}
	(i) 
		Recall that we have $ A_1(\Kur) \simeq \Ahat(\Kur) \simeq (\Gmhat(\Kur))^{\oplus g}$ as noted above. 
	The isomorphism \eqref{eq:A1} gives 
	the following commutative diagram:
	\begin{equation}
		\label{diag:Ahat}
		\vcenter{
	\xymatrix{
		(\wh{\mathbb{G}}_m(\Kur))^{\oplus g } & \ar@{_{(}->}[l](\mu_{p})^{\oplus g}\\
	A_1(\Kur) \ar[u]^{\simeq}  &\ar@{_{(}->}[l] A_1[p]\ar[u]^{\simeq}.
	}}
	\end{equation}	
%	\textcircled{$\star$}
	The above diagram makes the square $(\spadesuit)$ commutative in the next diagram. 
	\begin{equation}
	\label{diag:A1Kur}
		\vcenter{
	\xymatrix@C=2mm@R=0mm{
	A_1(\Kur)/p \ar@/^5ex/[rrrr]^-{\delta_{A_1}^{\Kur}}\ar[rr]^-{\simeq} \ar@{^{(}->}[dd]_-{\delta_{A_1}}&  &\left(\Gmhat(\Kur)/p\right)^{\oplus g}\ar[rr]^{\iota}\ar@{^{(}->}[dd]_-{\delta_{\Gmhat}} & & \left(\Kurt/p\right)^{\oplus g}  \ar[rr]^-{v}\ar@{=}[dd]^{\delta_{\Gm}} & & (\Z/p)^{\oplus g}\\ 
	& (\spadesuit) & & \\
	H^1(\Kur,A_1[p]) \ar[rr]^-{\simeq} & & H^1(\Kur,\Gmhat[p])^{\oplus g}\ar@{=}[rr] & & H^1(\Kur,\mu_p)^{\oplus g}, 
	}}
	\end{equation}
	where $\iota$ is induced from $\Gmhat(\Kur) = U_{\Kur}^1 \inj \Kurt$ 
	and $v$ is the valuation map. 
	Since 
	$\Gmhat(\Kur) = U^1_{\Kur}\subset O_{\Kur}^{\times}$, 
	$v\circ \iota = 0$ in the diagrm
	\eqref{diag:A1Kur}. 
	Hence, the claim $\Im(\delta_{A_1}^{\Kur}) \subset (\Ubar_{\Kur})^{\oplus g} = \Ker(v: (\Kur/p)^{\oplus g}\to (\Z/p)^{\oplus g})$ holds.
%	
%	\begin{equation}
%	\label{diag:AK}
%	\vcenter{
%	\xymatrix{
%	\left(\Gmhat(\Kur)/p\right)^{\oplus g}\ar[r]\ar@/^10ex/[rrrd]  & H^1(\Kur,\mu_{p})^{\oplus g}\ar[rd]  \\
%	A_1(\Kur)/p\ar[u]^-{\simeq} \ar[r]^-{\delta_{A_1}}\ar[d] & H^1(\Kur,A_1[p])  \ar[u]^-{\simeq}\ar[r]^{\simeq}\ar[d] & (\Kurt/p)^{\oplus g} \ar[r]^{v}\ar[d]^{\Id^{\oplus g}}  & (\Z/p)^{\oplus g}\ar[d]^{\Id^{\oplus g}} \\
%	A(\Kur)/p \ar[r]^-{\delta_{A}} & H^1(\Kur,A[p]) \ar[r]^{\simeq} & (\Kurt/p) ^{\oplus 2g} \ar[r]^-{v^{\oplus g}} & (\Z/p)^{\oplus 2g}. 
%	}}
%	\end{equation}

\sn
(ii) 
Consider the following short exact sequence 
\[
A_1(\Kur)/p \to A(\Kur)/p \to \Abar(\F_{\Kur})/p \to 0. 
\]
Since the residue field $\F_{\Kur} = \ol{\F}_{K}$ is algebraically closed, 
$\Abar(\F_{\Kur})/p = \Abar(\ol{\F}_K)/p = 0$ 
and hence the natural map $A_1(\Kur)/p \surj A(\Kur)/p$ is surjective. 
This map gives the commutative diagram below: 
\[
\xymatrix{
A(\Kur)/p \ar[r]^-{\delta_A^{\Kur}} & (\Kurt/p)^{\oplus 2g} \\
A_1(\Kur)/p \ar@{->>}[u]\ar[r]^-{\delta_{A_1}^{\Kur}} & (\Kurt/p)^{\oplus g} \ar[u]_{(\Id,1)}.
}
\]
From this diagram, the image of $\delta_A^{\Kur}$ is 
contained in $(\Ubar_{\Kur})^{\oplus g}$ as claimed.
\end{proof}

Next, we study the image of 
\[
\delta_{A_1}^K: A_1(K)/p\inj H^1(K,A_1[p]) \simeq H^1(K,\mu_p)^{\oplus g} = (\Kt/p)^{\oplus g},
\]

\begin{lem}
\label{lem:Tak2}
	$\Im(\delta_{A_1}^{K}) = (\Ubar_K)^{\oplus g}$.
\end{lem}
\begin{proof}
	First, we show the following claim:	
	\begin{claim}
		$\Im(\delta_{A_1}^{K}) \subset (\Ubar_{K})^{\oplus g}$.
	\end{claim}
	\begin{proof}
	Consider the following commutative diagram: 
	\[
	\xymatrix{
	A_1(K)/p \ar[r]^-{\delta_{A_1}^K}\ar[d] & (\Kt/p)^{\oplus g} \ar[r]^-{v}\ar[d]^-{j} & (\Z/p)^{\oplus g}\ar[d]^-{\Id}   \\ 
	A_1(\Kur)/p\ar[r]^-{\delta_{A_1}^{\Kur}} & (\Kurt/p)^{\oplus g} \ar[r]^-{v}  & (\Z/p)^{\oplus g}\\
	}
	\]
	From Lemma \ref{lem:Tak1} (i), 
	the composition $v\circ \delta_{A_1}^{\Kur} = 0$ in the above diagram. 
	The composition 
	$v\circ \delta_{A_1}^K = 0$ 
	in the top sequence and hence $\Im(\delta_{A_1}) \subset 
	(\Ubar_{K})^{\oplus g}$.
\end{proof}
Next, we compare the orders of 
both of $\Im(\delta_{A_1}^{K})$ and $(\Ubar_K)^{\oplus g}$. 
As $\Ubar_K = \Ker(v:\Kt/p \to \Z/p)$, 
the last claim indicates 
\[
\#A_1(K)/p \le \# (\Ubar_K)^{\oplus g} = p^{g([K:\Qp] + 1)}.
\]
On the other hand, 
	Mattuck's theorem (\cite{Mat55}) and the assumption $A[p]\subset A(K)$ say  
	\[
	A(K)/p \simeq (\Z/p)^{\oplus g([K:\Qp]+2)}. 
	\]
	The short exact sequence 
	\[
	A_1(K)/p \to A(K)/p \to \Abar(\F_K)/p \to 0
	\]
	and 
	$\#\Abar(\F_K)/p = \#\Abar(\F_K)[p] = p^{g}$ 
	as $\Abar$ has ordinary reduction, 
	we obtain the inequality 
	\[
	\#A_1(K)/p\ge p^{g([K:\Qp]+1)}.
	\]
	Therefore, 
	the map $\delta_{A_1}^{K}:A_1(K)/p \isomto (\Ubar_K)^{\oplus g}$ is bijective.
\end{proof}

\begin{proof}[Proof of Prop.~\ref{prop:Tak}]
\begin{claim}
	$\Im(\delta_A^K) \subset (\Ubar_K)^{\oplus g}\oplus \Ker \left(j:\Kt/p\to \Kurt/p\right)^{\oplus g}$.
\end{claim}
\begin{proof}
	Consider the following commutative diagram:
		\begin{equation}
		\label{diag:delta}
		\vcenter{
		\xymatrix@R=3mm{
		 A(K)/p \ar@{-->} '[rd]_{\varphi} [rrdd]\ar[r]^-{\delta_A^K} \ar[dd] & (\Kt/p)^{\oplus 2g} \ar[dd]^(.3){j}\ar[r]^{v} & (\Z/p)^{\oplus 2g}\ar[dd]^{\Id} \\
		 & & \\
		 A(\Kur)/p \ar[r]^-{\delta_{A}^{\Kur}} & (\Kurt/p)^{\oplus 2g} \ar[r]_{v} & (\Z/p)^{\oplus 2g},
	}}
	\end{equation} 
	From Lemma \ref{lem:Tak1} (ii), 
	the image of the composition $j\circ \delta_A^K$ is contained in $(\Ubar_{\Kur})^{\oplus g}$. 
	In particular,  
	the image of $\varphi$ the dotted arrow in the above diagram is $0$ so that 
	we obtain 
	$\Im(\delta_A^K) \subset (\Ubar_K)^{\oplus g}\oplus \Ker (j)^{\oplus g}$.
\end{proof}

	\noindent
	(i) From Mattuck's theorem (\cite{Mat55}) and the assumption $A[p] \subset A(k)$,  
	\[
	A(K)/p \simeq  (\Z/p)^{\oplus g([K:\Qp] + 2)}. 
	\] 
	On the other hand, 
	\[
	\Ker\left(\Kt/p \to \Kurt/p\right) \simeq \Ker\left(H^1(K,\mu_{p}) \to H^1(\Kur,\mu_{p})\right) \simeq H^1(\Kur/K,\mu_{p}) \simeq \Z/p, 
	\]
	and hence 
	\[
	\# \left(\Ubar_K \oplus \Ker\left(\Kt/p\to \Kurt/p\right) \right)^{\oplus g} 
	= \# (\Kt/p)^{\oplus g} = p^{g([K:\Qp]+2)}.
	\]
%	Here, the last equality follows from the group structure 
%	of $\Kt$ (\cite{Neu99}, Chap.~II, Prop.~5.7). 
	By counting the cardinality, 
	we obtain the equality 
	$\Im(\delta_A^K) = (\Ubar_K)^{\oplus g} \oplus \Ker(\Kt/p \to \Kurt/p)^{\oplus g}$.	
	
	\sn
	(ii) 
	From Lemma \ref{lem:Tak2} and the commutative diagram below give the assertion.
	\[
	\xymatrix{
	A_1(K)/p \ar[r]^{\delta_{A_1}^K} \ar[d] & (\Kt/p)^{\oplus g} \ar[d]^{(\Id,1)} \\
	A(K)/p\ar[r]^{\delta_A^K} & (\Kt/p)^{\oplus 2g}.
	}
	\]
%	The image of 
%	the composition 
%	\[
%	\Ahat(K) \simeq A_1(K) \onto{\delta_{A_1}} H^1(K,A_1[p]) \simeq (\Kt/p)^{\oplus g}
%	\]
%	is contained in $(\Ubar_K)^{\oplus g}$. 
%	In particular, $\# \Ahat(K)/p \le \#(\Ubar_K)^{\oplus g } = p^{g([K:\Qp]+1)}$. 
%	On the other hand, 
%	Mattuck's theorem says 
%	\[
%	A(K)/p \simeq (\Z/p)^{\oplus g([K:\Qp]+2)}
%	\]
%	and 
%	$\#\Abar(\F_K)/p = \#\Abar(\F_K)[p] = p^{g}$. 
%	The short exact sequence 
%	\[
%	\Ahat(K)/p \to A(K)/p \to \Abar(\F_K)/p \to 0
%	\]
%	gives $\#\Ahat(K)/p\ge p^{g([K:\Qp]+1)}$.
%	Therefore, we have the map $\Ahat(K)/p \simeq (\Ubar_K)^{\oplus g}$ is bijective.
\end{proof}
Now we define 
the sub Mackey functors $\U,\V \subset \Gm/p$ by 
	\begin{equation}
	\label{def:Ubar}		
	\U(K) := \Ubar_K = \Im(U_K\to \Kt/p),\quad \V(K) := \Ker\left(j:\Kt/p \to \Kurt/p\right), 
	\end{equation}
	for any finite extension $K/k$.  
	Note that, we have 
\[
\V(K) =  
 \Im\left(U_K^{pe_0(K)}\to \Kt/p\right),
\]
where $e_0(K) = e_K/(p-1)$ 
and $e_K$ is the ramification index of $K/\Qp$ (\Cf \cite{Tak11}, Rem.~3.2). 
In fact, 
both of the subgroups of $\Kt/p$ are annihilators of 
$\Ubar_K$ in the Hilbert symbol. 
By the fixed isomorphism \eqref{eq:isom}, 
the following diagram is commutative
\[
	\xymatrix{
  A(L)/p \ar@{^{(}->}[r]\ar@<0.5ex>[d]^{N_{L/K}} & H^1(L,A[p]) \ar@<0.5ex>[d]^{\Cor_{L/K}}\ar[r]^{\simeq} &  (L^{\times}/p)^{\oplus g}\ar@<0.5ex>[d]^{N_{L/K}}  & \ar@{_{(}->}[l]    \U(L)^{\oplus g} \oplus  \V(L)^{\oplus g}\,\ar@<.5ex>[d]\\ 
  A(K)/p \ar@{^{(}->}[r]\ar@<0.5ex>[u]^{\Res_{L/K}}&  H^1(K,A[p])\ar[r]^{\simeq}\ar@<0.5ex>[u]^{\Res_{L/K}} &   (\Kt/p)^{\oplus g}\ar@<0.5ex>[u] & \ar@{_{(}->}[l]    \U(K)^{\oplus g} \oplus  \V(K)^{\oplus g},\ar@<.5ex>[u]
  }
\]
for any finite extensions $L/K/k$.  
We obtain the following isomorphisms 
of Mackey functors: 

\begin{cor}
\label{cor:Tak}
	There are isomorphisms 
	\[
	A/p \simeq \U^{\oplus g} \oplus \V^{\oplus g}, 
\quad \mbox{and}\quad 
\Ahat/p \simeq \U^{\oplus g}
	\]
	of Mackey functors. 
\end{cor}

%%%%%%%%%%%%%%%%%%%%%%%%%%%%%%%%%%%%%
\section{Galois symbol map}
\label{sec:symbol}

Let $A$ be an abelian variety over $k$ which has 
good reduction.  
The formal group $\Ahat$ defines a Mackey functor by 
the associated group  
$\Ahat(K) := \Ahat(\mK)$ for a finite extension $K/k$. 
Note that we have $\Ahat(K)  \simeq A_1(K)$ ([HS], Thm.~C.2.6). 
There is a short exact sequence of Mackey functors 
\begin{equation}
\label{seq:E}
0\to \wh{A} \to A \to \As \to 0,
\end{equation}
where $\As  = A/\Ahat$ is defined by the exactness 
(in the abelian category of Mackey functors). 
The Mackey functor $\As$\, has the following description 
(\Cf \cite{RS00}, Sect.~3, (3.3)): 
For a finite extension $K/k$ with residue field $\FK$, 
\begin{equation}
\label{eq:Ebar}
	\As(K) \simeq \Abar(\FK).
\end{equation}
For finite extensions $L/K/k$ with ramification index $e(L/K)$, 
the restriction $\Res_{L/K}:\As(K)\to \As(L)$ 
and the norm map $N_{L/K}:\As(L)\to \As(K)$
can be identified with the restriction 
$\Abar(\F_L) \to \Abar(\FK)$ and 
$e(L/K)N_{\F_L/\FK}:\Abar(\F_L)\to \Abar(\FK)$ respectively. 

For each $n \ge 1$, we denote by 
\[
\Mhatn := (\Gm/p^n\otimesM \wh{A}/p^n)(k),\quad \mbox{and}\quad 
\Mbarn := (\Gm/p^n\otimesM \As/p^n)(k). 
\] 
By applying $\Gm/p^n\otimesM -$ (which is right exact) 
to the sequence \eqref{seq:E}, 
we have 
\begin{equation}
	\label{diag:MK}
	\vcenter{
\xymatrix{
& \Mhatn \ar[r]\ar@{->>}[d] &  (\Gm/p^n\otimesM A/p^n)(k) \ar[r]\ar@{->>}[d] & \Mbarn \ar[r] \ar@{-->>}[d] & 0\, \\
0 \ar[r] & \Khatn \ar[r] & K(k;\Gm,A)/p^n \ar[r] & \Kbarn \ar[r] & 0, 
}}
\end{equation}
where 
$\Khatn$ is the image of the composition $\Mhatn \to (\Gm/p^n\otimesM A/p^n)(k) \to K(k;\Gm,A)/p^n$ 
and $\Kbarn$ is defined by the exactness of the lower sequence.

\subsection*{Special fiber}
%Recall that we defined $\Mbarn = (\Gm/p^n\otimesM \As/p^n)(k)$ and 
%$\As = A/\Ahat$ in \eqref{seq:E}.

\begin{lem}
	\label{lem:Mbar}
	Let $A$ be an abelian variety over $k$ with good reduction. 
	Then, we have 
	\[
	\Mbarn =  (\Gm/p^n\otimesM \As/p^n)(k) \simeq \Abar(\Fk)/p^n
	\] 
	for any $n\in \Z_{\ge 1}$.
\end{lem}
\begin{proof}
	First, we define a Mackey functor $\mathscr{Z}$ by 
	$\mathscr{Z}(K) = \Z$ 
	for each finite extension $K/k$ and, 
	for a finite extensions $L/K/k$, 
	the norm $N_{L/K}:\mathscr{Z}(L) \to \mathscr{Z}(K)$ is multiplication by the residue degree $f({L/K})$ of $L/K$ 
	and the restriction $\Res_{L/K}:\mathscr{Z}(K) \to \mathscr{Z}(L)$  
	is multiplication by the ramification index $e({L/K})$ of $L/K$. 
	The valuations give a morphism 
	 $v:\Gm \to \mathscr{Z}$ of Mackey functors 
	 defined by the valuation map  
	 \[
	 v: \Gm(K) = K^{\times } \to \mathscr{Z}(K) = \Z,
	 \]
	 of $K$, for any finite extension $K/k$.  
	%Its kernel is denote by $\Ubar := \Im(v:\Gm\to \mathscr{Z})$. 
	%Note that $U(K) = \Ker(v_K:K^{\times}\to \Z) = O_K^{\times}$ is the unit group 
	%for any extension $K/k$. 
	As in \eqref{def:Ubar}, we define a Mackey functor $\U_n$ by 
	\[
	\U_n(K) := \Im\left(U_K \to \Kt/p^n\right)
	\]
	which fits into 
	the short exact sequence 
	\[
	0\to \U_n \to \Gm/p^n \onto{v} \mathscr{Z}/p^n \to 0.
	\]
	From the right exactness of $- \otimesM \As/p^n$ for $\As = A/\wh{A}$ (\Cf \eqref{eq:Ebar}), we have 
	\begin{equation}
	\label{seq:Ubar}
		\U_n \otimesM \As/p^n \to \Gm/p^n \otimesM \As/p^n \to \mathscr{Z}/p^n \otimesM \As/p^n\to 0.
	\end{equation}
	
	\begin{claim}
	$(\mathscr{Z}/p^n \otimesM \As/p^n)(k) \simeq \Abar(\Fk)/p^n$
	\end{claim}
	\begin{proof}
		By identifying $\As(K) \simeq \Abar(\FK)$ (\Cf \eqref{eq:Ebar}), define 
		\[
			\varphi: (\mathscr{Z}/p^n \otimesM \As/p^n)(k) \to 
				\Abar(\Fk)/p^n; 
			\set{a,x}_{K/k} \mapsto N_{\FK/\Fk}(ax)
		\] 
		and 
		\[
			\psi:\Abar(\Fk)/p^n \to 
				(\mathscr{Z}/p^n \otimesM \As/p^n)(k); 
				x \mapsto \set{1,x}_{k/k}.
		\] 
		It is easy to see that these maps are well-defined and $\varphi\circ \psi = \Id$, where $\Id$ is the identity map. 
		To show  $\psi \circ \varphi= \Id$, take an element of the form 
		$\set{a,x}_{K/k} \in (\mathscr{Z}/p^n \otimesM \As/p^n)(k)$. 
		Let $k'\subset K$ be the maximal unramified subextension of $K/k$. 
		We have 
		\begin{align*}
			\psi\circ \varphi(\set{a,x}_{K/k}) 
		  	&= \set{1,N_{\FK/\Fk}(ax)}_{k/k}\\
		  	&= \set{a,N_{k'/k}x}_{k/k} 
		  		\quad (\mbox{since $N_{k'/k} = N_{\FK/\Fk}:\As(K) = \Abar(\FK)\to \As(k) = \Abar(\Fk)$})\\
		  	&= \set{\Res_{k'/k}a,x}_{k'/k} \quad (\mbox{by the projection formula})\\
		  	&= \set{a,x}_{k'/k}\quad (\mbox{by $\Res_{k'/k}a = e(k'/k)a = a$})\\
		  	&= \set{N_{K/k'}a, x}\quad (\mbox{by $N_{K/k'}(a) = f(K/k')a = a$})\\
		  	&= \set{a,\Res_{K/k'}x}_{k'/k}\quad (\mbox{by the projection formula})\\
		   	&= \set{a,x}_{K/k} \quad (\mbox{since $\Res_{K/k'}= \Id:\As(k') = \Abar(\FK)\to \As(K) =\Abar(\FK) $}).
  		\end{align*}
	\end{proof}
	
	\begin{claim}
		$(\U_n \otimesM \As/p^n)(k) = 0$.
	\end{claim}
	\begin{proof}
		Recall that $U^1_k = 1+\mathfrak{m}_k \subset U_k = O_k^{\times}$ induces 
		$U_k/U^1_k \simeq \Fk^{\times}$. 
		The residue field $\Fk$ is finite, in particular, perfect, 
		so that we have  $\U_n(k) = \Im(U^1_k \to k^{\times}/p^n)$.  
		From the norm arguments, it is enough to show that 
		$\set{a,x}_{k/k} = 0$ in $(\U_n \otimesM \As/p^n) (k)$.
		For such an element $\set{a,x}_{k/k}$, 
		there exists a finite unramified extension $K/k$ such that 
		$\Res_{K/k}(x) = p^n\xi$ for some $\xi \in \As(K) \simeq \Abar(\FK)$. 
		Since the norm map $N_{K/k}:U_K^1\to U_k^1$ is surjective 
		(\cite{Ser68}, Chap.~V, Prop.~3), 
		one can find $\alpha \in \U_n(K)$ such that $N_{K/k}(\alpha) =a$. 
		From this, we obtain 
		\begin{align*}
			\set{a,x}_{k/k} &= \set{N_{K/k}(\alpha),x}_{k/k} \\
			&= \set{\alpha, \Res_{K/k}(x)}_{K/k} \quad (\mbox{by the projection formula})\\
			&= \set{\alpha, p^n\xi}_{K/k} = 0.
		\end{align*}
	\end{proof}
	The short exact sequence \eqref{seq:Ubar} and 
	the above claims yield the assertion. 
\end{proof}

\subsection*{Mackey product and the Somekawa $K$-group}
%Keep the notation and the hypotheses of the last section. 
We define 
\begin{equation}
\label{def:N}
		N :=\max \set{n\in \Z_{\ge 0} | A[p^n] \subset A(k)}. 
\end{equation}
In the following, we assume the abelian variety $A$ satisfies 
the condition \Rat in the last section.
From this assumption \Rat, $N\ge 1$. 
We fix isomorphisms 
$A[p^n] \simeq (\mu_{p^n})^{\oplus 2g}$ for all $n\le N$ as follows: 
First, take 
an isomorphism $A[p^N] \simeq (\mu_{p^N})^{\oplus 2g}$ which makes the 
diagram 
\[
\xymatrix@C=15mm{
A_1[p^N] \ar@{^{(}->}[r] \ar[d]^{\simeq} & A[p^N]\ar[d]^{\simeq} \\
(\mu_{p^N})^{\oplus g}  \ar@{^{(}->}[r]^-{(\Id,1)} & (\mu_{p^N})^{\oplus g}\oplus (\mu_{p^N})^{\oplus g}
}
\]
commutes as in \eqref{eq:isom}. 
For each $1\le n< N$, we choose $A[p^n]\simeq (\mu_{p^n})^{\oplus 2g}$ as the following diagram is commutative:
\begin{equation}
\label{diag:isoms}
	\vcenter{
\xymatrix{
A[p^{n}] \ar[d]_{\simeq}\ar@{^{(}->}[r]& A[p^{n+1}]\ar[d]^{\simeq}\\ 
(\mu_{p^{n}})^{\oplus 2g} \ar@{^{(}->}[r]& (\mu_{p^{n+1}})^{\oplus 2g}.
}}
\end{equation}

To simplify the notation, put 
\[
\M_n := (\Gm \otimesM A)(k)/p^n \simeq (\Gm/p^n \otimesM A/p^n)(k),\quad  \mbox{and}\quad \H_n :=H^2(k,\mu_{p^n}\otimes A[p^n]).
\]

\begin{thm}
	\label{thm:GmAV}
	For any $n \ge 1$, 
	the Galois symbol map 
	$s_{p^n}^M:\Mn \to \Hn$  (Def.~\ref{def:symbol})
%	\[
%	s_{p^n}^M:(\Gm \otimesM A)(k)/p^n \to H^2(k,\mu_{p^n}\otimes A[p^n])
%\]	
	 is bijective.
\end{thm}
\begin{proof}
The map $s_{p^n}^M$ 
is surjective (\cite{Gaz19}, Thm.~A.1). 
We show that $s_{p^n}^M$ is injective by induction on $n$. 

\begin{claim}
$s_p^M$ is bijective. 
\end{claim}
\begin{proof}
By the fixed isomorphism $A[p] \simeq (\mu_p)^{\oplus 2 g}$ 
as in \eqref{diag:isoms}, 
we have the isomorphism  
%\[
%A(k)/p \simeq (\Ubar_k)^{\oplus g} \oplus \left(\Ubar_k^{pe_0(k)}\right)^{\oplus g},
%\]
%where $\Ubar_k = \Im(U_k \to \kt/p)$ 
%and $\Ubar_k^{pe_0(k)} = \Im(U_k^{pe_0(k)} \to \kt/p)$ 
%(Prop.~\ref{prop:Tak} (i)). 
%Moreover, the isomorphism above 
%can be extend 
\[
A/p \simeq \U^{\oplus g} \oplus \V^{\oplus g} 
\]
of Mackey functors (Cor.~\ref{cor:Tak}). 
%where 
%$\Ubar$ is the Mackey functor defined by 
%$\Ubar(k') := \Ubar_{k'}$ and 
%$\Ubar^{pe_0}$ is the Mackey functor defined by 
%$\Ubar^{pe_0}(k') := \Ubar^{pe_0(k')}_{k'}$. 
The Mackey product $\M_1 = \Gm/p\otimesM A/p$ 
is decomposed as 
\[
\M_1 = \Gm/p\otimesM \left(\U^{\oplus g} \oplus \V^{\oplus g} \right) \simeq 
(\Gm/p \otimesM \U)^{\oplus g} \oplus 
(\Gm/p \otimesM \V)^{\oplus g}.
\]
%into the direct sum of 
%\[
%\Gm/p \otimesM \Ubar,\quad \mathrm{and}\quad \Gm/p\otimesM \Ubar^{pe_0}.
%\]
Since the composition 
\[
(\Gm/p\otimesM\V)(k) \to (\Gm/p\otimesM \U)(k) \to (\Gm/p\otimesM \Gm/p)(k) \simeq K_2^M(k)/p \simeq H^2(k,\mu_{p}^{\otimes 2})
\] 
is bijective (\cite{Hir16}, Thm.~3.6),
the Galois symbol map $s_p^M$ is also bijective 
from the commutative diagram below: 
\[
\xymatrix{
\M_1 \ar[d]^{\simeq} \ar[r]^{s_p^M} & \H_1 \ar[d]^{\simeq}  \\
(\Gm/p \otimesM \U)(k)^{\oplus g} \oplus 
(\Gm/p \otimesM \V)(k)^{\oplus g} \ar[r]^-{\simeq} & H^2(k,\mu_p^{\otimes 2})^{\oplus 2 g}.
}
\]
\end{proof}

%For $n\ge 1$, 
%put 
%	\[
%		\M_n := \left(\Gm/p^n  \otimesM A/p^n\right)(k),\quad \mbox{and},\quad \H_n := H^2(k, \mu_{p^n}\otimes A[p^n]). 
%	\]
We consider the following commutative diagram with exact rows 
	except possibly at $\M_n$ (\Cf \cite{RS00}, Proof of Lem.~4.2.2): 
	\[
		\xymatrix{
		k^{\times}\otimesZ A[p] \ar[r]^-{\psi}\ar[d]^{\phi} 
			& \M_n \ar[r]^p\ar@{^{(}->}[d]^{s_{p^n}^M}  
			& \M_{n+1} \ar[r]\ar[d]^{s_{p^{n+1}}^M} & \M_1\ar@{^{(}->}[d]^{s_{p}^M} \\
		H^1(k,\mu_{p}\otimes A[p]) \ar[r] & \H_n \ar[r]
			& \H_{n+1} \ar[r] & \H_1. 
	}
	\]
%%	The far left square is commutative up to sign. 
%%	In fact, it induces from 
%%	\[
%%	\xymatrix{
%%	A[p] \ar@{=}[d] \ar[r] &A(k)/p^n\ar[d]^{\delta}  \\
%%	H^0(k,A[p]) \ar[r]^{\delta} & H^1(k,A[p^n]).
%%	}
%%	\]
%%   (cf. \cite{NSW}, 1.4.3, p. 37)
	Here, 
%	the bottom horizontal sequence 
%	is the long exact sequence arising from 
%	$0 \to \mu_{p^n}\otimes A[p^n]\to \mu_{p^{n+1}}\otimes A[p^{n+1}] \to \mu_p\otimes A[p] \to 0$, 
%	and 
	the far left vertical map $\phi$ is given by 
	\[
		\kt\otimesZ A[p] \xrightarrow{\delta_{\Gm}\otimes \Id} 
			H^1(k,\mu_p)\otimesZ H^0(k,A[p]) \onto{\cup} H^1(k,\mu_p\otimes A[p])
	\]
	and $\psi$ is induced from $A[p]\inj A(k) \surj A(k)/p^n$. 
	By the fixed isomorphism $A[p]\simeq (\mu_{p})^{\oplus 2g}$ 
	of trivial Galois modules, 
	the map $\phi$ becomes 
	\[
		\kt\otimesZ A[p] \surj (\kt/p\otimesZ \mu_p)^{\oplus 2g} \simeq 
			H^1(k,\mu_p^{\otimes 2})^{\oplus 2g} \simeq H^1(k,\mu_p\otimes A[p]).
	\]
	In particular, $\phi$ is surjective. 
	From the diagram chase and the induction hypothesis, 
	$s_{p^{n+1}}^M$ is injective. 	
\end{proof}

%
%\begin{rem}
%	\label{thm:Som}
%	From a similar argmuments in the proof of Theorem~\ref{thm:GmAV}, 
%	we have the following assertion: 
%	For $i = 1,2$, 
%	let $G_i$  
%	be a semi-abelian variety 
% 	of the form $G_i= \Gm^{\oplus r_1}\oplus A_i$
% 	over $k$, 
%	where $A_i$ is 
%	an abelian variety 
%	of dimension $g_i$. 
%	We assume the conditions \Rat\ and \Ord\ for $A_i$. 
%%	\begin{enumerate}
%%		\item 
%	Then, the Galois symbol map 
%	\[
%	s_{p^n}^M: (G_1\otimesM G_2)(k)/p^n \to H^2(k,G_1[p^n]\otimes G_2[p^n]) 
%	\] 
%	is injective, 
%	for any $n\ge 1$ with $A_i[p^n] \subset A_i(k)$ $(i=1,2)$.  
%	In particular, 
%	\[
%	s_{p^n}: K(k;G_1,G_2)/p^n \to H^2(k,G_1[p^n]\otimes G_2[p^n]) 
%	\] is injective. 
%\end{rem}

\begin{cor}
\label{cor:Sconj}
 \begin{enumerate}
 	\item 
 	For any $n \ge 1$, 
 	we have $\Mn = (\Gm/p^n \otimesM A/p^n)(k) \simeq K(k;\Gm,A)/p^n$.
 	
 	\item 
 	For any $n \ge 1$, the Galois symbol map 
 	$s_{p^n}$ is bijective. 
 	
 	\item 
	For any $n\le N$, we have  
	$\M_n \simeq (\Z/p^n)^{\oplus 2g}$.
 \end{enumerate}
\end{cor}
\begin{proof}
	As $s_{p^n}^M$ factors through $s_{p^n}$, 
	the assertions (i) and (ii) follow from Theorem \ref{thm:GmAV}. 
	If $n\le N$, 
	we have \[
	K(k;\Gm,A)/p^n \isomto  \Hn  \simeq H^2(k,\mu_{p^n}^{\otimes 2})^{\oplus 2g} \simeq (\Z/p^n)^{\oplus 2g}.
	\]
	The assertion (iii) follows from (i).	
\end{proof}

Note that we have $\M_n\simeq K(k;\Gm,A)/p^n$ the middle vertical map in the diagram \eqref{diag:MK} is bijective (Cor.~\ref{cor:Sconj}) and hence we have $\Mbarn \simeq \Kbar_n$.
 
\begin{cor}
\label{cor:Kbar}
	$\Kbarn \simeq \Mbarn \simeq \Abar(\F)/p^n$.
\end{cor}
\begin{proof}
The latter isomorphism 
$\Mbarn \simeq \Abar(\F)/p^n$ follows from Lemma \ref{lem:Mbar}. 
\end{proof}

\subsection*{Formal groups}
Since the Mackey functor $\Ahat$ defined by the formal group of $A$ satisfies the Galois descent, 
we have the Galois symbol map (Def.~\ref{def:symbol}) of the form
\[
\shat_{p^n}^M: \Mhatn \simeq  (\Gm\otimesM \Ahat)(k)/p^n 
\to H^2(k,\mu_{p^n}\otimes \Ahat[p^n]) =: \Hhatn.
\]

\begin{lem}
	\label{lem:Khat}
	For $n\le N$, we have the following: 
	\begin{enumerate}
	\item The Galois symbol map 
	$\shat_{p^n}^M: \Mhatn \to \Hhatn$ 
	is bijective. 
	\item $\Mhatn \simeq \Khatn \simeq (\Z/p^n)^{\oplus g}$.
	\end{enumerate}
\end{lem}
\begin{proof}
	(i) 
	Consider the following commutative diagram with exact rows: 		
	\begin{equation}
		\label{eq:shatMpn}
		\vcenter{		
		\xymatrix{
		\Mhat_{n-1} \ar[r]^{p}\ar[d]^{\shat_{p^{n-1}}^M}  & 
			\Mhat_{n} \ar[r]\ar[d]^{\shat_{p^{n}}^M} & 
			\Mhat_1 \ar[d]^{\shat_{p}^M} \ar[r] & 0\,\\
		\Hhat_{n-1}\ar[r]& 
			\Hhatn \ar[r] & 
			\Hhat_1\ar[r] & 0.	
	}}
	\end{equation}
%	Fix an isomorphisms of the trivial Galois modules 
%	\[
%	\xymatrix{
%		\Ahat[p^n] \ar[d]^{\simeq} \ar@{^{(}->}[r]  & A[p^n] \ar[d]^{\simeq} \\
%	(\mu_{p^n})^{\oplus g}   \ar@{^{(}->}[r] & (\mu_{p^n})^{\oplus 2g}.
%	}
%	\]
%	as in \eqref{diag:Ahat}. 
	From the assumption $n\le N$,  
	we have 
	$\Hhatn \simeq H^2(k,\mu_{p^n}^{\otimes 2})^{\oplus g} \simeq (\Z/p^n)^{\oplus g}$. 
	By counting the orders, the left lower map in \eqref{eq:shatMpn} is injective. 
	By induction on $n$,  
	it is enough to show that 
	$\shat_{p}^M:\Mhat_1 \to \Hhat_1$ is bijective. 
	In this case, we have an isomorphism 
	\begin{equation}
		\label{eq:Ehat}
		\Ahat/p \isomto \U^{\oplus g}
	\end{equation}
	of Mackey functors 	(Cor.~\ref{cor:Tak}). 
	On the other hand, we have $\Hhat_1 \simeq H^2(k,\mu_{p}^{\otimes 2})^{\oplus g}$. 
	The assertion now reduced to showing that the 
	composition 
	\[
		(\Gm/p\otimesM \U)(k) \to (\Gm/p\otimesM \Gm/p)(k) \isomto H^2(k,\mu_{p}^{\otimes 2})
	\]
	is bijective. 
	This assertion follows from \cite{Hir16}, Theorem 3.6 (i). 		

	\noindent
	(ii) 
	The Galois symbol map $s_{p^n}:K(k;\Gm,A)/p^n \to \Hn = H^2(k,\mu_{p^n}\otimes A[p^n])$ is bijective from Theorem~ \ref{thm:GmAV} and 
	$\Hn \simeq H^2(k,\mu_{p^n}^{\otimes 2})^{\oplus 2g} \simeq (\Z/p^n)^{\oplus 2g}$. 
	From the fixed commutative diagrams in \eqref{diag:isoms}, 
	we have the following commutative diagram: 
	\[
	\xymatrix@R=5mm{
	\Hhat_n \ar[r]^-{\iota} \ar[d]^{\simeq}& \Hn\ar[d]^{\simeq} \\
	H^2(k,\mu_{p^n}^{\otimes 2})^{\oplus g} \ar@{^{(}->}[r] & H^2 (k,\mu_{p^n}^{\otimes 2})^{\oplus 2g}, 
	}
	\]
	where  $\Hhatn = H^2(k,\mu_{p^n}\otimes \Ahat[p^n])$.
	As the bottom horizontal map is the inclusion map, 
	the map $\iota$ above is injective. 
	Next, 
	we consider the commutative diagram extended from \eqref{diag:MK}: 
		\[
		\xymatrix@R=4mm{
		\Mhatn \ar@{->>}[dd]\ar[rd]_-{\shat_{p^n}^M}^{\simeq}\ar[rr] & &\Mn \ar'[d]_-{\simeq}[dd] \ar[rd]^-{s_{p^n}^M} \\
		& \Hhatn \ar@{^{(}->}[rr]^(.4){\iota} & & \Hn \\ 
		\Khatn \ar@{^{(}->}[rr] & & K(k;\Gm,A)/p^n\ar[ru]_{\simeq}^{s_{p^n}}&.
		}
	\]
%		We choose an isomorphism $E[p^n] \simeq (\Z/p^n)^{\oplus 2}$ 
%	as the image of $\Ehat[p^n]$ in $E[p^n]$ maps to the first factor of $(\Z/p^n)^{\oplus 2}$. We have the following diagram:
%	\[
%	\xymatrix{
%	\Hhatn \ar[d]^{\simeq} \ar[r] & H^2(k,\mu_{p^n}\otimes E[p^n]) \ar[d] ^{\simeq}\\
%	\Z/p^n \ar@{^{(}->}[r]^{i} & (\Z/p^n)^{\oplus 2} .
%	}
%	\]
%	where $i$ is defined by $x \mapsto (x,0)$.
%	From the above diagram, the composition $\Khatn \to  K(k;\Gm,A)/p^n \to H^2(k,\mu_{p^n}\otimes A[p^n])$ 
%	factors through $\Hhatn$ (as the dotted arrow in the above diagram). 
	Since $\shat_{p^n}^M:\Mhatn\to \Hhatn$ is bijective, 
	the left vertical map in the above diagram is bijective. 
	Therefore, the induced map $\Khatn\to \Hhatn$ is also bijective. 
	\end{proof}

	\begin{prop}
	\label{prop:Khat}
		We further assume the following condition:
		\begin{description}
		\item [\Ram] $k(\mu_{p^{N+1}})/k$ is a non-trivial and totally ramified extension. 
	\end{description} 
	Then, we have $\Mhatn \simeq \Khatn \simeq (\Z/p^{\min\set{N,n}})^{\oplus g}$ for any $n \ge 1$. 
	\end{prop}
	\begin{proof} % Cf \cite{GL}, pf of Thm. 3.15
	From Lemma \ref{lem:Khat}, we may assume $n \ge N$. 
	First, we show the following claim: 
	\begin{claim}[\Cf \cite{GL18}, Proof of Thm.~3.14]
		Fix $\zeta\in \mu_{p^N}$ a primitive $p^N$-th root of unity. 
		Then, 
		$\Mhat_1 = (\Gm/p \otimesM \Ahat/p) (k)$ is generated by 
		elements of the form $\set{\zeta,w}_{k/k}$, where $w \in \Ahat(k)/p$. 
	\end{claim}
	\begin{proof}
	Recall that the Hilbert symbol $(-,-):k^{\times}\otimes k^{\times} \to \mu_p\simeq \Z/p$ satisfies  
	\begin{equation}
	\label{eq:H}
	(x,y) = 0\Leftrightarrow y \in N_{k(\sqrt[p]{x}\,)/k}\left(k(\sqrt[p]{x}\,)^{\times}\right),\quad \mbox{for $x,y\in k^{\times}$}
		\end{equation}
	(\Cf \cite{Tat76}, Prop.~4.3). 
	From the assumption \Ram\!, 
	there exists $y\in \Ubar_k =  \Im(U_k\to \kt/p)$ such that 
	$(\zeta,y) \neq 0$. 
	In fact, putting 
	$K = k(\sqrt[p]{\zeta}) = k(\mu_{p^{N+1}})$, 
	we have $U_k/N_{K/k}U_K \simeq k^{\times}/N_{K/k}K^{\times}$ (\Cf the proof of \cite{Ser68}, Chap.~V,\ Sect.~3,\ Cor.~7) and 
	local class field theory says 
	$k^{\times}/N_{K/k}K^{\times} \simeq \Gal(K/k)\neq 0$  (\Cf \cite{Ser68}, Chap.~XIII,\ Sect.~3). 
	Thus, there exists $y\in U_k \ssm N_{K/k}U_K$ such that 
	$(\zeta,y) \neq 0$ from \eqref{eq:H}. 
	As $(\zeta,y)\neq 0$, the chosen element $y$ is non-trivial in $\Ubar_k$. We use the same notation $y$ 
	as an element in $\Ubar_k$. 
	For each $1\le i \le g$, put 
	\[
	y^{(i)} := (1,\ldots , 1, \stackrel{i}{\stackrel{\vee}{y}}, 1,\ldots , 1) \in (\Ubar_k)^{\oplus g}
	\] 
	and we denote by $w^{(i)}\in \Ahat(k)/p$ 
	the element corresponding to $y^{(i)}$ through the isomorphism 
	$\Ahat(k)/p \simeq (\Ubar_k)^{\oplus g}$ as in \eqref{eq:Ehat}.
	The Galois symbol map is compatible with the Hilbert symbol map 
	(\cite{Ser68}, Chap.~XIV, Sect.~2, Prop.~5)
	as the following commutative diagram indicates: 
	\begin{equation}
	\label{diag:H}
	\vcenter{
	\xymatrix{
		\kt/p \otimesZ \Ahat(k)/p\ar[r]^-{\iota}\ar[d]^{\simeq} & \Mhat_1\ar[r]_-{\simeq}^-{\shat_{p}^M} \ar[r]& \Hhat_1 \ar[d]^-{\simeq} \\
\left(\kt/p \otimesZ \Ubar_k\right)^{\oplus g} \ar[rr]^-{(-,-)}& &(\Z/p)^{\oplus g}.
	}}
	\end{equation}
	Here, we identify 
	the isomorphism $H^2(k,\mu_p^{\otimes 2}) \simeq \Z/p$, 
	the map $\shat_p^M$ is bijective (Thm.\ \ref{thm:GmAV}),
	and the map $\iota$ is given by 
	$\iota( x\otimes w) = \set{x,w}_{k/k}$. 
	The image of $\zeta\otimes w^{(i)} \in \kt/p\otimesZ \Ahat(k)/p$ in $(\Z/p)^{\oplus g}$ 
	via the lower left corner in \eqref{diag:H} 
	is 
	\[
	\xi^{(i)} := (0,\ldots , 0, \stackrel{i}{\stackrel{\vee}{(\zeta,y)}}, 0, \ldots , 0)
	\in (\Z/p)^{\oplus g}. 
	\] 
	These elements $\xi^{(i)}$ $(1\le i \le g)$ generate $(\Z/p)^{\oplus g}$ 
	and hence the symbols $\set{\zeta,w^{(i)}}_{k/k} = \iota (\zeta \otimes w^{(i)})$ ($1\le i \le g$) 
	generate $\Mhat_1$. 
	\end{proof}
	
	There is a short exact sequence 
	\[
	\Mhat_1 \onto{p^n}  \Mhat_{n+1} \to \Mhatn \to 0, 
	\]
	where  the first map is given by the 
	multiplication by $p^n$.
%	$\set{\ol{x},\ol{y}}_{k'/k} \mapsto p^n\set{\ol{x},\ol{y}}_{k'/k}$. 
%	In fact, this is given by 
%	\[
%	(\Gm\otimesM\Ahat)/p (k) \onto{p^n} (\Gm\otimesM\Ahat)/p^{n+1} (k) \to (\Gm\otimesM\Ahat)/p^{n} \to 0
%	\]
%	The first map is $\ol{\set{x,y}_{k'/k}}\mapsto \ol{p^n\set{x,y}_{k'/k}}$.
	By Lemma \ref{lem:Khat}, we have $\Khat_N \simeq (\Z/p^N)^{\oplus g}$. 
	From the above claim and the induction on $n\ge N$, 
	we have $\Mhat_{n+1} \simeq \Mhatn \simeq (\Z/p^N)^{\oplus g}$. 
	The surjective homomorphisms 
	\[
	\Mhatn \surj \Khatn \surj \Khat_{N} \simeq (\Z/p^N)^{\oplus g}
	\] 
	induce $\Khatn \simeq (\Z/p^N)^{\oplus g}$. 
 \end{proof}

\begin{rem}
	In the case where $A = E$ is an elliptic curve with \Ord,\ 
	we can show Proposition 3.6 by the essentially same proof,\ under the condition 
	\begin{description}
	\item[\Ramp] $k(\Ehat[p^{N+1}])/k$ is a non-trivial extension
	\end{description}
	instead of \Ram\!. 
%	When the reduction $\Ebar$ 
%	satisfies $\Ebar[p^{N+1}] \subset \Ebar(\Fk)$, 
%	the short exact sequence 
%	\[
%	0 \to \mu_{p^{N+1}}= \Ehat[p^{N+1}] \to E[p^{N+1}]\to \Ebar[p^{N+1}] =\Z/p^{N+1}\to 0
%	\] exists. 
%%	This gives a long exact sequence 
%%	\[
%%	0 \to (\mu_{p^{N+1}})^{G_k} \to E[p^{N+1}]^{G_k}\to \Ebar[p^{N+1}]^{G_k} = \Ebar[p^{N+1}] \to H^1(k,\mu_{p^{N+1}}).
%%	\]
%%	By counting the cardinality, 
%	Thus the condition \Ramp above implies that 
%	$k(\mu_{p^{N+1}})/k$ is a non-trivial extension.
%	Conversely, 
%	if we assume 
%	\begin{itemize}[label=$\circ$]
%		\item $\Ebar[p^{N+1}] \subset \Ebar(\Fk)$, and 
%		\item $k(\mu_{p^{N+1}}) \supsetneq k$, 
%	\end{itemize}
%	then the condition \Ramp\ holds.
\end{rem}

\subsection*{Proof of the main theorem}

\begin{thm}
\label{thm:main}
	Let $A$ be an abelian variety over $k$ of dimension $g$. 
	We assume the condition \Ord, \Rat\,and\, \Ram\ (in Prop.~\ref{prop:Khat}). 
	Then, we have
	\[
	K(k;\Gm,A)_{\fin} \simeq (\Z/p^N)^{\oplus g} \oplus \Abar(\Fk).
	\]
\end{thm}

\begin{proof}
	For each $m$ prime to $p$, 
	we have $K(k;\Gm,A)/m\simeq \Abar(\F)/m$ (Prop.~\ref{prop:Bloch2}). 
	
	\begin{claim}
		 $K(k;\Gm,A)/p^n \simeq \Khatn \oplus \Kbarn$ for each $n\ge 1$. 
	\end{claim}
	\begin{proof}
		The assertion is true for $n\le N$ by Theorem~\ref{thm:GmAV} and Lemma \ref{lem:Khat}. 
		%	In fact, we have $\Z/p^n \surj\Ebar(\Fk)/p^n \simeq \Mbarn \surj \Kbarn \simeq \Z/p^n$.
		For $n>N$, consider the following diagram:
		\[
		\xymatrix{
		0 \ar[r]& \Khatn \ar[r] & K(k;\Gm,A)/p^n \ar[r] & \Kbarn \ar[r] & 0 \\
		0 \ar[r]& \Khat_{n+1} \ar[u]\ar[r] & K(k;\Gm,A)/p^{n+1} \ar[u]\ar[r] & \Kbar_{n+1} \ar[u]\ar[r] & 0. 
		}
		\]
		By the induction hypothesis, the top sequence splits. 
		From Proposition~\ref{prop:Khat}, the left vertical map is bijective 
		so that the lower sequence also splits. 
	\end{proof}
%	The bijection 
%	$\left(\Gm/p^n \otimesM E/p^n\right)(k) \simeq K(k;\Gm,E)/p^n$ (Lem.~\ref{lem:MS}) 
%	leads $\Mbarn \simeq \Kbarn$. 
	From 
	Cororally \ref{cor:Kbar} and Proposition \ref{prop:Khat}, 
	we obtain 
	\[
	K(k;\Gm,A)/p^n \simeq (\Z/p^{\min\set{n,N}})^{\oplus g} \oplus \Abar(\Fk)/p^n
	\]
	for any $n$.
	By taking the limit $\plim$, we have
	$K(k;\Gm,A)_{\fin} \simeq (\Z/p^N)^{\oplus g}\oplus \Abar(\F)$.
\end{proof}

Applying the above theorem to the Jacobian variety $\Jac(X)$, 
we obtain the following corollary as noted in Introduction.
\begin{cor}
	\label{cor:main}
	Let $X$ be a projective smooth curve over $k$ with 
	$X(k)\neq \emptyset$.  
	We assume the conditions 
	\Rat,\ \Ord,\ and \Ram\ for the Jacobian variety $J = \Jac(X)$ associated to $X$. 
	Then, we have 
	\[
	\VXfin \simeq (\Z/p^{N})^{\oplus g} \oplus \ol{J}(\Fk), 
	\]
	where $g = \dim J$ and $N = \max\set{ n | J[p^n] \subset J(k)}$.
\end{cor}

\appendix

\section{Albanese kernel}

In this appendix, 
we compute the Albanese kernel for the product of curves over $k$ 
as a generalization of \cite{Tak11}.  
Let $A$ be an abelian variety over $k$ assuming 
\Ord and \Rat. 
For $n\ge 1$ with $A[p^n]\subset A(k)$, 
we choose an isomorphism 
\[
	A[p^n] \isomto (\mu_{p^n})^{\oplus 2g}
\]
of (trivial) Galois modules 
as in \eqref{eq:isom}. 
By the same proof as in Prop \ref{prop:Tak}, 
the image of the Kummer map $\delta_A: A(k)/p^n \to H^1(k,A[p^n]) \simeq H^1(k,\mu_{p^n}^{\oplus 2g}) \simeq (\kt/p^n)^{\oplus 2g}$ is 
determined as follows: 
\begin{prop}
For any $n \ge 1$ with $A[p^n] \subset A(K)$ and a finite extension $K/k$, we have
	\[
	A(K)/p^n \isomto (\Ubar_K)^{\oplus g} \oplus \Ker\left(j:\Kt/p^n\to \Kurt/p^n\right)^{\oplus g}, 
	\] 
	where $j$ is the map induced from the inclusion $\Kt \inj \Kurt$. 
\end{prop}

The above isomorphism is extend to the isomorphism 
\[
	A/p^n \simeq \U^{\oplus g} \oplus \V^{\oplus g}
\]
of Mackey functors, 
where $\U,\V$ is the sub Mackey functors of $\Gm/p^n$ defined by 
	\[
	\U(K) := \Im(U_K\to K^{\times}/p^n),\quad \V(K) := \Ker\left(j:\Kt/p^n \to \Kurt/p^n\right), 
	\]
	for any finite extension $K/k$ (\Cf Cor.~\ref{cor:Tak}). 
Now we use the following notation: 
\begin{itemize}
	\item $X_i$\,: smooth projective curves over $k$ with 
$k$-rational point $X_i(k)\neq \emptyset$ for $i=1,2$, 
	\item $J_i := \Jac(X_i)$: the Jacobian variety associated to $X_i$ of dimension $g_i$. 
\end{itemize}
The kernel of the degree map $\deg:\CH(X_1\times X_2)\to \Z$ 
is denoted by $A_0(X_1\times X_2)$. 
The kernel $T(X_1\times X_2)$ of the Albanese map 
\[
\operatorname{alb} : A_0(X_1 \times X_2) \to \operatorname{Alb}_{X_1\times X_2}(k) = J_1(k)\oplus J_2(k)
\] 
which is called the \textbf{Albanese kernel} 
is also written by the Somekawa $K$-group as 
\[
T(X_1\times X_2) \simeq K(k;J_1,J_2) 
\]
(\cite{RS00}). 
From the same computation as in \cite{Tak11}, Theorem 4.1, 
we recover \cite{Gaz18}, Corollary 8.9. 
\begin{cor}
Assume that $J_1,J_2$ satisfy \Ord 
and all $p^n$-torsion points of $J_1$ and $J_2$ are $k$-rational. Then, we have 
\[
T(X_1\times X_2)/p^n \simeq (\Z/p^n)^{\oplus g_1g_2},
\]
where $g_i = \dim J_i$.
\end{cor}
\begin{proof}
	From $K(k;J_1,J_2) \simeq T(X_1\times X_2)$, 
	it is enough to show 
	$K(k;J_1,J_2)/p^n \simeq (\Z/p^n)^{\oplus g_1g_2}$. 
	We fix isomorphisms $J_i[p^n] \simeq (\mu_{p^n})^{\oplus 2g_i}$ 
	of (trivial) Galois modules as above. 
	We have 
	\[
	J_1/p^n\otimesM J_2/p^n \simeq 
%	(\U^{\oplus g_1} \oplus \V^{\oplus g_1}) \otimes M 
%	(\U^{\oplus g_2} \oplus \V^{\oplus g_2})
	\left( (\U\otimesM \U) 
	\oplus (\U\otimesM \V) 
	\oplus (\V\otimesM \U)
	\oplus (\V\otimesM \V)
	\right)^{\oplus g_1g_2}.	
	\]
	The Galois symbol maps give the following commutative diagram: 
	\begin{equation}
		\label{eq:J}
	\vcenter{
	\xymatrix@R=1mm@C=5mm{
	(J_1/p^n\otimesM J_2/p^n)(k)\ar@/^30pt/[rr]^{s_{p^n}^M}\ar[dd]^{\simeq}  \ar@{->>}[r]&  K(k;J_1,J_2)/p^n \ar@{^{(}->}[r]^-{s_{p^n}} & H^2(k,J_1[p^n]\otimes J_2[p^n]) \ar[dd]^{\simeq} \\ 
	& \\ 
	(\U\otimesM \U)(k)^{\oplus g_1g_2} \ar@{}@<-2ex>[d]^{\bigoplus} & &
	H^2(k,\mu_{p^n}^{\otimes 2})^{\oplus 4g_1g_2}\\ 
	 (\U\otimesM \V)(k)^{\oplus 2g_1g_2}\ar@{}@<-2ex>[d]^{\bigoplus}\ar@/_10pt/[rru]^{s_{p^n}^M}  & \\
	(\V\otimesM \V)(k) ^{\oplus g_1g_2}.
	}}
	\end{equation}
	Here, 
	$s_{p^n}:K(k;J_1,J_2)/p^n \to H^2(k,J_1[p^n]\otimes J_2[p^n])$ is injective (\cite{RS00}, Rem.~4.5.8 (b)), 
	and 
	the bottom map 
	is given by composing 
	$s_{p^n}^M:(\Gm/p^n\otimesM \Gm/p^n)(k)\to H^2(k,\mu_{p^n}^{\otimes 2})$ (Def.~\ref{def:symbol}) 
	after the natural map 
	$(\U\otimesM \U)(k) \to (\Gm/p^n\otimesM \Gm/p^n)(k)$, 
	$(\U\otimesM \V)(k) \to (\Gm/p^n\otimesM \Gm/p^n)(k)$, 
	or $(\V\otimesM \V)(k) \to (\Gm/p^n\otimesM \Gm/p^n)(k)$. 
	The Galois symbol map is written by the Hilbert symbol (\cite{Ser68}, Chap.~XIV, Sect.~2, Prop.~5) as follows:  
	
	\[
	\entrymodifiers={!! <0pt, .8ex>+}
	\xymatrix@R=3mm{
	(\U\otimesM \U)(k) \ar[r] & (\Gm/p^n\otimesM \Gm/p^n)(k)\ar[r]^-{s_{p^n}^M} & H^2(k,\mu_{p^n}^{\otimes 2}) \\ 
	\displaystyle \bigoplus_{K/k} \U(K) \otimesZ \U(K) \ar[r] \ar@{->>}[u]\ar@{-->}@/_30pt/[rrd] & \displaystyle \bigoplus_{K/k} K^{\times}/p^n \otimesZ K^{\times}/p^n \ar[r]^-{\delta \cup \delta }\ar@{->>}[u]\ar@/_10pt/[dr]^-{(-, -)} & \displaystyle \bigoplus_{K/k} H^2(K,\mu_{p^n}^{\otimes 2})\ar[u]^{\Cor_{K/k}}\\ 
	& & \displaystyle \bigoplus_{K/k} \Z/p^n\ar[u]_-{\simeq}.
	}
	\]
	For each finite extension $K/k$, the image of 
	$\U(K)\otimes_{\Z} \U(K)$ by the Hilbert symbol 
	(the dotted arrow in the above diagram) is isomorphic to 
	$\Z/p^n$ (\cite{Tak11}, Prop.~2.5) so that 
	\[
	s_{p^n}^M\left((\U\otimesM\U)(k)\right) \simeq \Z/p^n.
	\]
	On the other hand, 
	the image of 
	$\U(K)\otimes_{\Z} \V(K)$ and $\V(K)\otimes_{\Z} \V(K)$ 
	by the Hilbert symbol is trivial 
	and hence 
	\[
	s_{p^n}^M\left((\U\otimesM\V)(k)\right) = 
	s_{p^n}^M\left((\V\otimesM\V)(k)\right) = 0.
	\]
	From the above diagram \eqref{eq:J}, 
	these computations imply 
	\[
	s_{p^n}^M\left((J_1/p^n\otimesM J_2/p^n)(k)\right)  \simeq K(k;J_1,J_2)/p^n \simeq (\Z/p^n)^{\oplus g_1g_2}.
	\]
\end{proof}
%\bibliographystyle{amsplain}
%\bibliography{cft_p}

\def\cprime{$'$}
\providecommand{\bysame}{\leavevmode\hbox to3em{\hrulefill}\thinspace}
\providecommand{\MR}{\relax\ifhmode\unskip\space\fi MR }
\providecommand{\MRhref}[2]{%
  \href{http://www.ams.org/mathscinet-getitem?mr=#1}{#2}
}
\providecommand{\href}[2]{#2}

\end{document}